\documentclass[a4paper,11pt]{amsart}
\usepackage{amssymb,amsthm,amsmath,psfrag,cases}
\usepackage{mathrsfs}

\usepackage[utf8]{inputenc}
\usepackage[T1]{fontenc}

\def\dual#1#2#3{{\langle #1, #2 \rangle_{H^{-1/2}(#3),H^{1/2}(#3)}}}

\def\into{\longrightarrow}
\def\inside{{\rm insd}}
\let\phi=\varphi

\DeclareMathOperator{\cotanh}{cotanh}

\newtheorem{theorem}{Theorem}[section]
\newtheorem{lemma}[theorem]{Lemma}
\newtheorem{proposition}[theorem]{Proposition}
\newtheorem{corollary}[theorem]{Corollary}
\theoremstyle{definition}

\newfont\rsfseleven{rsfs10 scaled 1100}

\def\CalD{\mbox{\normalfont\rsfseleven D}}

\theoremstyle{remark}
\newtheorem{remark}[theorem]{Remark}
\numberwithin{equation}{section}

\def\curl{\text{curl}}
\def\eps{\varepsilon}
\let\epsilon=\varepsilon
\let\phi=\phi
\def\into{\longrightarrow}

\begin{document}

\title[Lagrangian controllability of fluids]{Lagrangian controllability of\\  inviscid incompressible fluids: \\
a constructive approach.}
\thanks{The authors wish to thank O. Glass, G. Legendre and F.X. Vialard for fruitful discussions and particularly on numerical issues.}

%%    author one information
 \author{T. Horsin}
 \address{CNAM, Laboratoire M2N EA7340, 292 rue Saint-Martin, Case 2D5000, F-75003 Paris }
 \curraddr{}
 \email{thierry.horsin@lecnam.net}

%%    author two information
 \author{O. Kavian}
 \address{Universit\'e de Versailles Saint-Quentin; Laboratoire de Mathématiques de Versailles (UMR 8100); 45 avenue des Etats-Unis; 78030 Versailles cedex; France}
 \curraddr{}
 \email{kavian@math.uvsq.fr}
% \thanks{The authors wish to thank O. Glass, G Legendre and F.X. Vialard for fruitfull discussions, and particularly the two }
 
 \subjclass[2000]{Primary }
 

%\dedicatory{}

\date{May 10, 2016}

\begin{abstract}
We present here a constructive method of Lagrangian approximate controllability for the Euler equation. We emphasize on different options that could be used for numerical recipes: either,  in the case of a bi-dimensionnal fluid, the use of formal computations in the framework of explicit Runge approximations of holomorphic functions by rational functions, or an approach based on the study of the range of an operator by showing a density result. For this last insight in view of numerical simulations in progress, we analyze through a simplified problem the observed instabilities.
\end{abstract}
\maketitle

%%%%%%%%%%%%%%%%%%%%%%%
\section{Introduction and main results}\label{sec:Intro}

\noindent Let $\Omega \subset {\Bbb R}^N$, with $N \geq 2$, be a bounded domain  with a regular boundary $\partial\Omega$, and let $\Gamma$ be a part of $\partial \Omega$ with nonempty relative interior.
\begin{figure}[!ht]
\begin{center}
\resizebox{!}{3.5cm}{\includegraphics{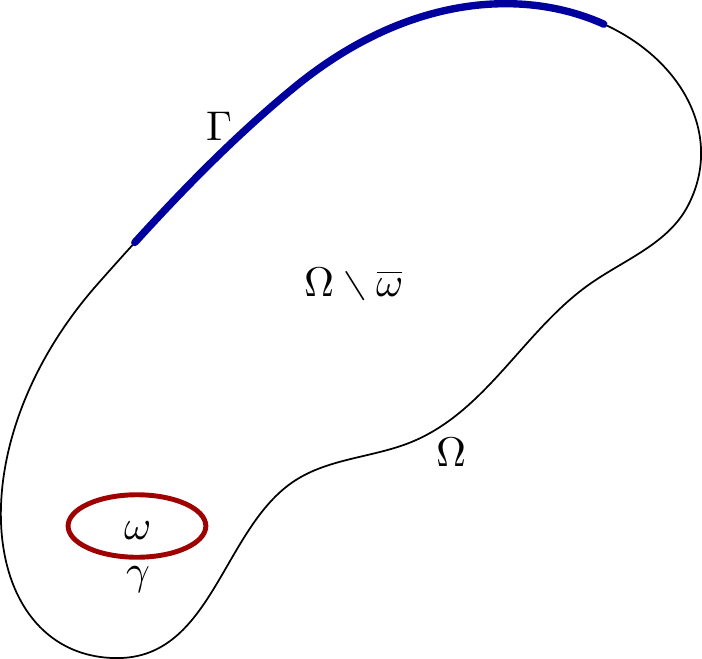}}
\end{center}
%\caption{}
\end{figure}
Assume that a subdomain $\omega \subset \subset \Omega$ is given such that its boundary $\gamma := \partial\omega$ is a Jordan curve and let us denote by ${\bf n}$ the exterior normal to the boundary of $\Omega \setminus \overline{\omega}$.  The question we address in this paper is the following: given a function $h$ defined on $\gamma$, can one find a function $v$ defined on $\partial\Omega$ having its support ${\rm supp}(v)\subset \Gamma$, and such that the solution $\Psi$ of
\begin{equation}\label{eq:Psi-v-1}
\Delta\Psi = 0\quad\mbox{in }\, \Omega, \qquad {\partial\Psi \over \partial{\bf n}} = v\quad\mbox{on }\, \partial\Omega,
\end{equation}
satsifies 
\begin{equation}\label{eq:Psi-v-h}
{\partial\Psi \over \partial{\bf n} } = h\quad \mbox{on }\, \gamma\mbox{ ?}
\end{equation}
The motivation of this question lies in its application to the Lagrangian control of Euler equation. Indeed, if such a $v$, and thus such a $\Psi$ exist, then upon considering a function $h$ depending smoothly on $t \in [0,T]$ for some $T > 0$, one may reasonably expect that $v$ and $\Psi$ might also depend smoothly on $t$, and therefore, upon setting 
$$u := \nabla \Psi, \qquad p := -\partial_{t}\Psi - {1 \over 2}|\nabla\Psi|^2,$$
the pair $(u,p)$ is a solution of the Euler equation
\begin{subnumcases}{\label{eq:euler} ~}
\partial_t u +(u\cdot \nabla) u + \nabla p = 0 \quad \text{ in }\; (0,T)\times \Omega,\label{eq:euler.1}\\ 
{\rm div}(u)  = 0 \quad \text{ in }\; (0,T)\times \Omega,\label{eq:euler.2}\\
u(0,\cdot)  = u_0 \quad  \text{ in }\; \Omega,\label{eq:euler.3}\\
u\cdot{\bf n} = 0 \quad  \text{ on }\; (0,T)\times (\partial \Omega\setminus \Gamma),\label{eq:euler.4}
\end{subnumcases}
where in addition we have $u\cdot {\bf n} = v$ on $\Gamma$, and the value of the normal component of $u(t,\cdot)$ on $\gamma$ is prescribed, that is $u(t,\cdot)\cdot {\bf n} = h(t,\cdot)$  on $\gamma$. From this point of view, one can say that a control problem is solved by the means of the mapping $(h,\gamma) \mapsto v$.
This is precisely the Lagrangian control of \eqref{eq:euler}, as investigated by O. Glass \& T. Horsin in \cite{GlHo08} and \cite{GlHo10}. As a matter of fact, proving the Lagrangian controllability is a consequence of the fact that one may prescribe the velocity of a certain set of fluid particles, so that its topological and regularity properties along its motion are preserved. With this approach of the problem, it is then enough to prescribe the normal velocity of this set of particles at every point of its boundary.  This is the motivation of our first result.
\medskip

Before stating the first result of this paper, let us recall briefly the following definitions and notations. A set $\gamma \subset {\Bbb R}^N$ with $N = 2$ or $N = 3$ is called a Jordan curve (when $N = 2$) or a Jordan surface (when $N = 3$) if one has $\gamma = \Phi(S^{N-1})$ where $\Phi : S^{N-1} \into {\Bbb R}^N$ is a continuous and injective mapping. Then it is known that ${\Bbb R}^N\setminus \gamma$ has exactly two connected components, one of them being bounded, which will be denoted by $\inside(\gamma)$ (the {\it inside\/} of $\gamma$).
A smooth (resp. analytic) Jordan curve or surface corresponds to the case where in addition $\Phi$ is smooth (resp. analytic).

Let $\Omega \subset {\Bbb R}^N$ be a smooth bounded domain, and let $\gamma \subset \subset \Omega$ be a smooth Jordan curve  or  surface. We shall denote by  $\Omega_2:=\inside(\gamma)$ the {\it inside\/} of $\gamma$ (see above), and by $\Omega_1:=\Omega\setminus\overline{\Omega_2}$ its complement. Also we will denote by ${\bf n}_{12}$ the unit normal vector on $\gamma$ pointing from $\Omega_1$ towards $\Omega_2$, and naturally we will denote ${\bf n}_{21}=-{\bf n}_{12}$, the normal pointing from $\Omega_{2}$ into $\Omega_{1}$. As usual we will denote by $H^{1/2}(\gamma)$ the space of traces on $\gamma$ of functions in $H^1(\Omega_2)$, which coincides with the traces on $\gamma$ of functions in $H^1(\Omega_1)$, since $\gamma$ is sufficiently smooth. We will denote by  $H^{-1/2}(\gamma)$ the dual of $H^{1/2}(\gamma)$, the duality between the two being denoted by $\langle\cdot,\cdot\rangle$, or $\langle\cdot,\cdot\rangle_{H^{-1/2}(\gamma),H^{1/2}(\gamma)}$ if it is necessary to avoid ambiguities.  Also we will denote by $H^{-1/2}_{m}(\gamma)$ the orthogonal of the constants in $H^{-1/2}(\gamma)$:
$$H^{-1/2}_m(\gamma):=\left\{v\in H^{-1/2}(\gamma) \; ; \; \langle v,1\rangle = 0\right\}.$$
We are given $\Gamma$, a closed connected part of $\partial \Omega$ with a non empty relative interior in $\partial \Omega$, and we will denote 
$$H^{-1/2}_m(\Gamma):=\left\{v\in H^{-1/2}(\partial\Omega)\; ; \; v=0\text{ in }\CalD'(\partial\Omega\setminus \Gamma),\,\mbox{ and }\, \langle v,1\rangle=0\right\}.$$
Our first result is the following:

\begin{theorem}\label{lem:Im-Dense}
Let $\Omega\subset {\Bbb R}^N$ be a smooth bounded domain, $\gamma$ be a Jordan curve or surface in $\Omega$, and let $\Gamma$ be a closed, connected part of $\partial\Omega$ with non empty relative interior.
For $v \in H^{-1/2}_{m}(\Gamma)$ denote by $\Psi_{v} \in H^1(\Omega)$  the unique solution of
\begin{subnumcases}{\label{eq:PsivDeLapin} ~} %\label{eq:PsivDeLapin}
-\Delta \Psi_v=0\text{ in }\Omega,\label{eq:Psivharm}\\
\dfrac{\partial \Psi_v}{\partial {\bf n}}=v\text{ on }\partial \Omega,\\
\int_\Omega\Psi_v dx=0,
\end{subnumcases}
and define the operator $\Lambda_{\gamma} : H^{-1/2}_m(\Gamma) \into  H^{-1/2}_m(\gamma)$ by setting
\begin{equation}\label{eq:DefLambdaGamma}
\Lambda_{\gamma}v := \nabla\Psi_v\cdot{\bf n}_{12|\gamma}.
\end{equation}
Then the operator $\Lambda_{\gamma}$ has a dense image in $H^{-1/2}_{m}(\gamma)$.
\end{theorem}

This means that given $h \in H^{-1/2}_{m}(\gamma)$, while in general it is not possible to find $v \in H^{-1/2}_{m}(\Gamma)$ such that \eqref{eq:Psi-v-1} and \eqref{eq:Psi-v-h} are satisfied (due for example to the mere fact that $\Psi$, solution to \eqref{eq:Psi-v-1}, is analytic in $\Omega$), nevertheless given any $\epsilon > 0$ one may find $v \in H^{-1/2}_{m}(\Gamma)$ such that 
\begin{equation}\label{eq:Approx-v-h}
\mbox{there exists }\, \Psi := \Psi_{v}\;\mbox{satisfying \eqref{eq:PsivDeLapin} and }\quad \|h - \partial\Psi/\partial{\bf n}\|_{H^{-1/2}} < \epsilon.
\end{equation}
Thus, using the above Theorem one may solve a Lagrangian approximate control of Euler equation, since in general an exact control is not possible.

In fact in section \ref{sec:Runge}, as far as the dimension $N = 2$ is concerned, we give another insight about the fact that the operator $\Lambda_{\gamma}$ has a dense image, by using a complex variable approach. Indeed, using the classical Runge theorem (see for instance \cite{Rudin} and section \ref{sec:Runge}), it is possible to give a procedure for the construction of an appropriate $v$ such that \eqref{eq:Approx-v-h} is fulfilled, through an expansion in series.
\medskip

Going back to the motivations of the above Theorem, let us recall that the Lagrangian control problem is the following: given a time interval $[0,T]$, a family of smooth subdomains $\omega_{t} \subset \subset \Omega$ depending continuously on $t\in [0,T]$ (in a sense yet to be precised), and a  function $h : [0,T]\times \Omega \into {\Bbb R}$, can one find a solution $(u,p)$ of the Euler system \eqref{eq:euler} such that for all $t\in [0,T]$ one has $u(t,\cdot)\cdot {\bf n} = h(t,\cdot)$ on $\partial\omega_{t}$? We should point out that in the system \eqref{eq:euler} the initial data $u_{0}$ is given, and the solution $u$ is determined through appropriate boundary datas on $\Gamma$ ensuring the existence and uniqueness of the solution $u$, those boundary datas playing the role of the control on $\Gamma$.

As it is observed by O. Glass \& T. Horsin in \cite{GlHo08} and \cite{GlHo10}, in general this control problem does not have a solution, essentially due to some restrictions intrinsically imposed by the vorticity $\nabla\wedge u$ when $u$ is a solution to \eqref{eq:euler}.

However if the above problem is relaxed into an approximate control problem, a positive answer has been given in the references previously quoted, provided some restrictions are imposed on the subdomains $\omega_{t}$ and on the function $h$. 

Recall that if $\gamma_{0}$ and $\gamma_{1}$ are merely continuous images of $\mathbb{S}^{N-1}$, they are said to be homotopic in $\Omega$, when there exists $g \in C([0,1]\times \mathbb{S}^{N-1} ; \Omega)$ such that $g(j,\cdot)$ is a parameterization of $\gamma_{j}$ for $j=0$ and $j=1$; moreover if $\gamma_1$ is reduced to a point, then one says that $\gamma_0$ is contractible in $\Omega$. The following result is proved in \cite{GlHo08} (here and in the sequel we denote by $\mathrm{meas}(\omega)$ the Lebesgue measure of a measurable set $\omega \subset {\Bbb R}^N$):

\begin{theorem}\label{th:lagrcontr2}
Assume that $N=2$, that $\gamma_0$ and $\gamma_1$ are hotomopic smooth Jordan curves in $\Omega$, and that, denoting by $\inside(\gamma)$ the {\em inside\/} of $\gamma$, the following conditions are satisfied:
$$\mathrm{meas}(\inside(\gamma_0)) = \mathrm{meas}(\inside(\gamma_1)),\qquad
\mbox{and}\qquad u_0\in C^\infty(\Omega).$$ 
Then approximate controllability between $\gamma_0$ and $\gamma_1$ holds at any time $T>0$ in $L^\infty$-norm, in the sense described below in Theorem \ref{th:lagrcontrpot}. 
\end{theorem}

A similar result, but with stronger assumptions on $\gamma_{0},\gamma_{1}$, is given in \cite{GlHo10} when $N=3$.

The cornerstone of the proof of theorem \ref{th:lagrcontr2} relies on the resolution of the controllability question in the case when $u_0=0$, and the construction of a vector field $X$ satisfying \eqref{eq:champX} (A remark concerning the existence of such $X$ is given further). Indeed in that case, the following result is proved in \cite{GlHo08} and \cite{GlHo10}, which motivates the approach of this paper.

\begin{theorem}\label{th:lagrcontrpot}
Let either $N=2$ and the assumptions of theorem \ref{th:lagrcontr2} be satisfied, or let $N=3$ and assume that $\gamma_0$ and $\gamma_1$ are smooth Jordan surfaces, contractible in $\Omega$. Then, given $\eps>0$ and any vector field $X\in C^\infty([0,T]\times \Omega,{\Bbb R}^N)$ satisfying equations \eqref{eq:champX}, there exists $\delta >0$ and a function 
$\psi\in C^\infty([0,T]\times{\overline \Omega},{\Bbb R})$ such that 
\begin{subequations}\label{eq:lagrcontpot}
\begin{gather}
\psi(0,\cdot)=\psi(T,\cdot)=0\label{eq:lagrcontpot.1}\\
\forall t\in [0,T],\quad \Delta_x \psi(t,\cdot) = 0\label{eq:lagrcontpot.2}\\
\nabla_x\psi(t,\cdot)\cdot{\bf n}=0 \quad\mbox{ on }\partial \Omega\setminus \Gamma\label{eq:lagrcontpot.3}
\end{gather}
\end{subequations} and such that, if for each $t\in[0,T]$ we denote by $\gamma_t:=\Phi^X(0,t,\gamma_0)$ then we have
\begin{equation}\| \nabla_x \psi\cdot {\bf n}-X\cdot{\bf n}\|_{L^\infty(\gamma_t)}\leq \delta,\label{eq:deriveenormaldephiprochedeX} \end{equation}
\begin{equation}\Phi^{\nabla_x\psi}(0,\cdot,\cdot)([0,T]\times\gamma_0) \subset \Omega,\label{eq:flotdansOmegabis}\end{equation}  and 
\begin{equation}\|\Phi^{\nabla_x\psi}(0,T,\gamma_0)-\gamma_1\|_\infty\leq \eps.\label{eq:controlagpotok}\end{equation}
\end{theorem}
The precise meaning of this result is that, up to the construction of the vector field $X$, when $u_0 \equiv 0$ in \eqref{eq:euler.3}, one can obtain the approximate Lagrangian controllability of \eqref{eq:euler} between $\gamma_0$ and $\gamma_1$ in time $T$ by means of potential flows.
\medskip

The use of such a potential flow is related to the so-called {\it return method\/}, introduced by J.M~Coron~\cite{92COR}, which involves an appropriate change of scale in time, and is used when dealing with the case $u_0 \not\equiv 0$. 
\medskip

Let us point out that estimate \eqref{eq:deriveenormaldephiprochedeX} is necessary to obtain the approximate controllability, but in itself it does not imply readily \eqref{eq:controlagpotok}. 
\medskip

As a matter of fact, one can consider the approximate Lagrangian controllability from two different perspectives.  The first one is related to the problem of approximately extending harmonic functions, by an appropriate resolution of some elliptic equations, and the use of trace operators acting on spaces such as $H^{\pm 1/2}(\Gamma)$. The second point of view, only in the case  of dimension $N = 2$, consists in a constructive approach using Runge's approximation theorem, as treated in Section \ref{sec:Runge}.

\bigskip
In the sequel we shall adopt the following notations and conventions: for a sufficiently regular vector field $X:[0,T]\times \Omega\into \mathbb{R}^N$, let $\Phi^X$ denote the flow of $X$ defined by:
\begin{subequations}
\begin{gather}\Phi^X:[0,T]\times [0,T]\times \Omega\into \mathbb{R}^N,\\
\begin{cases}
\dfrac{\partial \Phi^X}{\partial t}(s,t,x)=X(t,\Phi^X(s,t,x)), & (s,t,x)\in [0,T]^2\times \Omega,  \\
\Phi^X(s,s,x)=x.
\end{cases}
\end{gather}
\end{subequations}
For instance one may assume that $X$ is uniformly Lipschitz on $[0,T]\times \Omega$, to ensure that $\Phi^X$ exists for all $(s,t,x) \in [0,T]^2\times \Omega$.

Let us consider $\gamma_0$ and $\gamma_1$, two smooth Jordan surfaces or curves  homotopic in $\Omega$. We assume that there exists a smooth vector field $X:[0,T]\times \Omega\into \mathbb{R}^N$ such that

\begin{subequations}
\label{eq:champX}
\begin{align}
&X(0,\cdot)\equiv X(T,\cdot) \equiv 0 \\
&X(t,\sigma)=0,\quad \mbox{for all }\, (t,\sigma)\in [0,T]\times \partial \Omega\\
&\Phi^X(0,t,\gamma_{0}):=\Phi^X(0,t,\cdot)(\gamma_0)\subset \Omega\quad\mbox{for all }\, t\in[0,T] \\
&\Phi^X(0,T,\gamma_{0}):=\Phi^X(0,T,\cdot)(\gamma_0)=\gamma_1\label{eq:champX3}\\
&{\rm div}(X(t,\cdot))=0\; \mbox{ in }\,\Omega,\,\mbox{for all }\, t\in [0,T].
\end{align}
\end{subequations}

Given a parameterization of $\gamma_0$, equality \eqref{eq:champX3} means that the image of $\gamma_{0}$ by $\Phi^X(0,T,\cdot)$ is a parameterization of $\gamma_1$.

In \cite{GlHo08} and \cite{GlHo10},  O.~Glass \& Th.~Horsin construct explicitly smooth vector fields $X$ satisfying the conditions \eqref{eq:champX} according to the specific assumptions on $\gamma_0$ and $\gamma_1$, which depend on the dimension $N=2$ or $N=3$. Despite having explicit procedures for the construction of the vector fields $X$, from a numerical analysis perspective, it is nevertheless necessary to understand the stability of such procedures.

\medskip 
The remainder of this paper is organized as follows. In Section \ref{sec:Im-dense} we prove Theorem \ref{lem:Im-Dense}, while in Section \ref{sec:Runge} a constructive method, based on Runge's theorem, is presented which applies only in dimension $N=2$.
In Section \ref{sec:Ill-posedness} we give a precise analysis of a Cauchy problem on the boundary for the Laplace operator. In fact we show that the stability constant is of the order $\exp({\rm dist}(\Gamma,\gamma))$, where $\Gamma$ is the region of the boundary on which a control is implemented, and $\gamma$ denotes the boundary of the region which one desires to control. Thus, in order to have a tractable numerical procedure for the Lagrange controllability, it is necessary that the zone $\Gamma$ should be close enough to $\gamma$. 

%%%%%%%%%%%%%%%%%%%%%%%

%%%%%%%%%%%%%%%%%%%%%%%
\section{Approximate controllability in $H^{-1/2}$-norm}\label{sec:Im-dense}

As we mentioned in the introduction, motion of curves in ${\Bbb R}^2$, or surfaces in ${\Bbb R}^3$, is governed by the dynamics of the normal velocity. 
In this chapter we prove first Theorem \ref{lem:Im-Dense} and then we comment (see Remark \ref{rem:2-3} below) how this theorem yields an approximate controllability result for the Euler equation, albeit in a weak sense, that is in $H^{-1/2}$-norm.

It is clear that the solution $\Psi_{v}$ of the system \eqref{eq:PsivDeLapin} exists, is unique, and the mapping $v \mapsto \Psi_{v}$ is continuous from $H^{-1/2}_{m}(\Gamma)$ into $H^1(\Omega)$. Thanks to a result due to J.L.~Lions (see \cite[chapitre VII, \S~5]{62Lions}), $\Lambda_{\gamma}(v):=\nabla\Psi_{v}\cdot {\bf n}_{12}$ is well defined on $\gamma$, and the operator $\Lambda_{\gamma}$ is continuous from $H^{-1/2}_{m}(\Gamma)$ into $H^{-1/2}_{m}(\gamma)$.
\medskip

Let us recall that $H^{1/2}_m(\gamma)\hookrightarrow L^2_m(\gamma)\hookrightarrow (H^{1/2}_m(\gamma))'$  with dense and compact imbeddings where we have set
$$ L^2_m(\gamma):=\left\{\phi \in L^2(\gamma),\, \int_\gamma \phi d\sigma=0\right\},$$
and as a matter of fact one has
$$\left(H^{1/2}_m(\gamma)\right)'=H^{-1/2}_m(\gamma).$$

In order to prove Theorem \ref{lem:Im-Dense}, we will introduce the following operators $\Lambda_{i}$ for $i=1,2$, defined through the resolution of an appropriate partial differential equation in $\Omega_{i}$ (recall that we have set $\Omega_2:=\inside(\gamma)$, the {\it inside\/} of $\gamma$, and that $\Omega_1:=\Omega\setminus\overline{\Omega_2}$). The operator $\Lambda_{1}$ is a Poincaré-Steklov type operator (also called a Neumann-to-Dirichlet operator), and is given by
\begin{equation}\label{eq:Def-Lambda-1}
\begin{split}
\Lambda_{1}:H^{-1/2}_m(\gamma)&\into H^{1/2}_m(\gamma)\\
\psi &\mapsto \xi(\psi)_{|\gamma}
\end{split}
\end{equation}
where $\xi := \xi(\psi)\in H^1(\Omega_1)$ is the unique solution of
\begin{subnumcases}{\label{eq:deftheta1} ~}
-\Delta \xi = 0 \quad \mbox{in }\, \Omega_1,\\
\dfrac{\partial \xi}{\partial {\bf n}_{12}} = \psi \quad \mbox{on }\,\gamma,\label{eq:theta1vautvsurgamma}\\
\dfrac{\partial \xi}{\partial {\bf n}} = 0 \quad \mbox{on }\,\partial \Omega 
\label{eq:dernormtheta1surGamma}\\
\int_\gamma \xi(\sigma)\,d\sigma = 0.
\end{subnumcases}
The operator $\Lambda_{2}$ is also a Poincaré-Steklov type operator, and is given by
\begin{equation}\label{eq:Def-Lambda-2}
\begin{split}
\Lambda_{2}:H^{1/2}_m(\gamma)&\into H^{-1/2}_m(\gamma)\\
\phi&\mapsto \dfrac{\partial \zeta(\phi)}{\partial {\bf n}_{21}} := \nabla \zeta(\phi)\cdot {{\bf n}_{21}}_{|\gamma}
\end{split}
\end{equation}
where $\zeta := \zeta(\phi)\in H^1(\Omega_2)$ is the unique solution of
\begin{equation}\label{eq:deftheta2}
\begin{cases}
-\Delta \zeta = 0 \quad\mbox{in }\,\Omega_2,\\
\zeta = \phi \quad\mbox{on }\,\gamma.
\end{cases}
\end{equation}
It is readily seen that the operators $\Lambda_{1} : H^{-1/2}_{m}(\gamma) \into H^{1/2}_{m}(\gamma)$ and $\Lambda_{2} : H^{1/2}_{m}(\gamma) \into H^{-1/2}_{m}(\gamma)$ are bounded and self-adjoint operators, and that with the above notations we have, for $\psi \in H^{-1/2}_{m}(\gamma)$ and $\phi \in H^{1/2}_{m}(\gamma)$,
%\begin{equation}
\begin{eqnarray}
\dual{\psi}{\Lambda_{1}\psi}{\gamma} = \int_{\Omega_{1}}|\nabla\xi(x)|^2\,dx, \label{eq:Rel-Op-Lambda-1}  \\
\dual{\Lambda_{2}\phi}{\phi}{\gamma} = \int_{\Omega_{2}}|\nabla\zeta(x)|^2\,dx.
\label{eq:Rel-Op-Lambda-2}
\end{eqnarray}

In the course of our proof of Theorem \ref{lem:Im-Dense} we shall need to show that the operator $T_{12} := I + \Lambda_{1}\Lambda_{2} $ is a homeomorphism on the space $H^{1/2}_{m}(\gamma)$. 
To this end we recall the following result: if $A$ and $B$ are two $n\times n$ self-adjoint semi-definite matrices, it is a well known result that all the eigenvalues of the matrix $AB$ are real and nonnegative, and as a consequence for any $\lambda > 0$ the matrix $I + \lambda AB$ is invertible. An analogous result holds for operators in acting in a Hilbert space, as stated in the following lemma.

\begin{lemma}\label{lem-AB-accr}
Let $H$ be a Hilbert space, with a scalar product and norm denoted respectively by $(\cdot|\cdot)$ and $\|\cdot\|$, and two bounded nonnegative (in the sense of forms) selfadjoint operators $A$ and $B$ defined in $H$. Then for any $\lambda > 0$ the operator $I + \lambda AB$ is invertible and has a bounded inverse.
\end{lemma}

\begin{proof}
We are going to verify that the kernel $N(I + \lambda AB) = \{0\}$ and that the range $R(I + \lambda AB)$ is closed. Recall that since $B = B^*$ is nonnegative in the sense of forms, in particular we have the Cauchy-Schwarz inequality stating that for any $u,v \in H$,
$$|(Bu|v)| \leq (Bu|u)^{1/2}\, (Bv|v)^{1/2}.$$
In particular note that if $u\in H$ is such that $(Bu|u) = 0$ then $Bu = 0$. If $u\in H$ is such that 
$$u + \lambda AB u = 0,$$
then by taking the scalar product of the above with $Bu$, and using the facts that $B$ is self-adjoint, and $A,B$ are nonnegative, we have
$$0 = (Bu|u) + \lambda (ABu|Bu) \geq (Bu|u) \geq 0,$$
yielding that $(Bu|u) = 0$, and thus $Bu = 0$. Since $u + \lambda ABu = 0$, this shows that $u = 0$ and thus $N(I + \lambda AB)= \{0\}$.

We show now that $R(I + \lambda AB)$ is closed. Indeed if $u_{n}, f_{n} \in H$ and $f \in H$ are such that
$$u_{n} + \lambda AB u_{n} = f_{n} \to f \quad \mbox{in }\, H$$
we set $g_{n,k} := f_{n} - f_{k}$ and $v_{n,k} := u_{n} - u_{k}$ so that
$$v_{n,k} +\lambda AB v_{n,k} = g_{n,k}.$$
We may take the scalar product of this equality with $Bv_{n,k}$ and obtain
\begin{align}
\|B^{1/2}v_{n,k}\|^2 = (Bv_{n,k}|v_{n,k}) & \leq
	(Bv_{n,k}|v_{n,k}) + \lambda (ABv_{n,k}|Bv_{n,k}) = (g_{n,k}|Bv_{n,k}) \nonumber\\
	& \leq (Bg_{n,k}|g_{n,k})^{1/2}\|B^{1/2}v_{n,k}\| \leq \|B\|^{1/2}\|g_{n,k}\|\, \|B^{1/2}v_{n,k}\|.\nonumber
\end{align}
We conclude that 
$$\|B^{1/2}u_{n} - B^{1/2}u_{k}\| \leq \|B\|\,\|f_{n} - f_{k}\|,$$
proving that $(B^{1/2}u_{n})_{n}$ is a Cauchy sequence, and therefore the sequence $(ABu_{n})_{n}$ is also a Cauchy sequence, the linear operator $AB^{1/2}$ being continuous. Thus there exists a certain $g \in H$ such that $ABu_{n} \to g$ as $n \to \infty$. Finally, if we set $u := f - \lambda g$, we have that $u_{n} \to u$ as $n\to\infty$, and also $ABu_{n} \to ABu$ and therefore $u + \lambda ABu = f$, that is $u \in R(I +\lambda AB)$, and the range of the operator $I + \lambda AB$ is closed.
\medskip

It is clear that, changing the roles played by $A$ and $B$, we can also see that 
$$N((I + \lambda AB)^*) = N(I + \lambda BA) = \{0\}$$
and that $R((I + \lambda AB)^*) = R(I + \lambda BA)$ is closed. Since by the closed range theorem of S. Banach (see for instance K. Yosida \cite[p. 205, chapter VII, \S 5]{YosidaFA}), we have 
$$R(I + \lambda AB) = N((I + \lambda AB)^*)^\perp = H,$$ 
we conclude that $(I + \lambda AB)$ is one-to-one, that is $(I + \lambda AB)^{-1}$ exists. Thanks for instance to an applications of Banach's closed graph theorem to the mapping $(I + \lambda AB)^{-1}$ (see K. Yosida \cite[p. 79, Theorem 1, chapter II, \S 6]{YosidaFA}), we infer that $I + \lambda AB$ has a bounded inverse and thus it is a homeomorphism of $H$ into itself.
\end{proof}

We are now in a position to prove the following result:
\begin{proposition}\label{prop:T12bij}
The map 
\begin{eqnarray*}T_{12}:H^{1/2}_m(\gamma)&\into &H^{1/2}_m(\gamma)\\
\phi&\mapsto&\phi+\Lambda_{1}\Lambda_{2}\phi
\end{eqnarray*}
is a one-to-one homeomorphism.
\end{proposition}

\begin{proof}
Let $J := H^{1/2}_{m}(\gamma) \into H^{-1/2}_{m}(\gamma)$ be the duality isomorphism given by F. Riesz' theorem. Then setting $A := \Lambda_{1}J$ and $B := J^{-1}\Lambda_{2}$, it is easily seen that $A$ and $B$ are two self-adjoint, nonnegative and bounded operators on $H^{1/2}_{m}(\gamma)$. 
For instance let us check that $A = A^*$ and is nonnegative. To simplify notations, set $H := H^{1/2}_{m}(\gamma)$ so that $H' = H^{-1/2}_{m}(\gamma)$. The isomorphism $J$ satisfies, for any $\phi_{1},\phi_{2} \in H$ and $\psi_{1},\psi_{2} \in H'$,
$$\langle \psi_{1}, \phi_{1} \rangle_{H',H} = (J^{-1}\psi_{1}|\phi_{1})_{H}, \qquad 
(\phi_{1}|\phi_{2})_{H} = (J\phi_{1}|J\phi_{2})_{H'} = 
\langle J\phi_{1},\phi_{2}\rangle_{H',H}.$$
Thus, for $\phi_{1},\phi_{2} \in H^{1/2}_{m}(\gamma)$, if we set $\psi_{k} := J\phi_{k}$ for $k=1,2$, using the above properties of $J$ and the fact that $\Lambda_{1} : H' \into H$ is selfadjoint, we have
\begin{align}
(\phi_{1}|A\phi_{2})_{H} &= (J^{-1}\psi_{1}|\Lambda_{1}\psi_{2})_{H} = \langle \psi_{1},\Lambda_{1}\psi_{1}\rangle_{H',H} = \langle \psi_{2},\Lambda_{1}\psi_{1} \rangle_{H',H}\nonumber\\
& = (J^{-1}\psi_{2}|\Lambda_{1}\psi_{1})_{H} = (\phi_{2}|\Lambda_{1}\psi_{2})_{H} \nonumber\\
& = (\phi_{2}|A\phi_{1})_{H} = (A\phi_{1}|\phi_{2})_{H}, \nonumber
\end{align}
which means that $A$ is selfadjoint (recall that $A$ is bounded). Setting $\psi := J\phi$, for $\phi \in H$, the fact that $A$ is nonnegative is a consequence of \eqref{eq:Rel-Op-Lambda-1} and the equality 
$$(A\phi|\phi)_{H} = \langle \psi,\Lambda_{1}\psi \rangle_{H',H} = \int_{\Omega_{1}}|\nabla\xi(x)|^2\,dx,$$
where $\xi$ satisfies \eqref{eq:deftheta1}.

It is clear that $T_{12} = I + AB$, and thus applying Lemma \ref{lem-AB-accr} we conclude that $T_{12}$ is a homeomorphism on $H^{1/2}_{m}(\gamma)$.
\end{proof}

\subsection{Proof of Theorem \ref{lem:Im-Dense}}
\begin{proof}
By a result due to S. Banach, it is well known that the closure of the range of $\Lambda_{\gamma}$ is the orthogonal of the kernel of its adjoint $\Lambda_{\gamma}^*$, that is $\overline{R(\Lambda_{\gamma})}= N(\Lambda_{\gamma}^*)^\perp $ (see for instance K. Yosida \cite[p. 205, chapter VII, \S 5]{YosidaFA}). Thus we have to show that $N(\Lambda_{\gamma}^*) = \{0\}$. This will be done in two steps.

\noindent {\bf Step 1.} In this step, we consider the following case. Assume that ${\widetilde \Omega}\subset {\Bbb R}^N$ is a domain such that $\omega \subset \subset {\widetilde \Omega}$ is connected, $\Gamma = \partial\omega$ is smooth, and finally $\Omega = {\widetilde \Omega}\setminus {\bar \omega}$ (see figure 1). 
We have to show that $\Lambda_{\gamma}^*$ is injective (recall that the operator $\Lambda_{\gamma}$ is defined in \eqref{eq:DefLambdaGamma}), and to do so we need to characterize this adjoint operator by establishing a certain representation formula.
\begin{figure}[!ht]
\begin{center}
\resizebox{!}{5cm}{\includegraphics{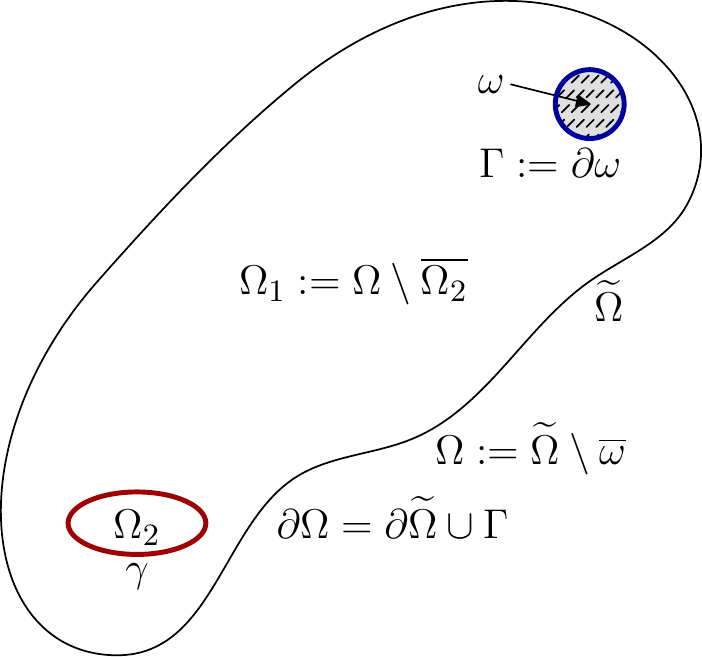}}
\end{center}
\caption{The case of the first step.}
\end{figure}

For $\phi\in H^{1/2}_m(\gamma)$ let us determine $\Lambda_{\gamma}^*(\phi)$. We solve \eqref{eq:deftheta2} and denote its solution by $\zeta := \zeta(\phi)$, and we denote by $\xi := \xi(-\Lambda_{2}(\phi))$ the solution of \eqref{eq:deftheta1} with $v = -\Lambda_{2}(\phi)$.
For this $v\in H^{-1/2}_m(\Gamma)$ given, multiply \eqref{eq:deftheta2} by $\Psi_{v}$ defined in \eqref{eq:PsivDeLapin}, and integrate by parts to obtain successively:
\begin{align}
\langle\Lambda_{\gamma} v,\phi\rangle_{H^{-1/2}(\gamma),H^{1/2}(\gamma)} &= 
- \int_{\Omega_2}\Delta \Psi_v(x) \,\zeta(x)\,dx 
- \int_{\Omega_2}\nabla \Psi_v(x)\cdot \nabla \zeta(x)\, dx  \nonumber\\
& = \displaystyle 
\int_{\Omega_{2}}\Psi_{v}(x)\,\Delta \zeta(x)\, dx 
- \langle\Lambda_{2}\phi,\Psi_{v}\rangle_{H^{-1/2}(\gamma),H^{1/2}(\gamma)}  \nonumber\\
& = 
\dual{\dfrac{\partial \zeta}{\partial{{\bf n}_{12}}}}{\Psi_{v}}{\gamma}   \label{eq:dual1}\\
& = 
\int_{\Omega_{1}}\Delta \xi(x)\, \Psi_{v}(x)\,dx 
+ \int_{\Omega_{1}}\nabla  \xi(x)\cdot \nabla \Psi_{v}(x)\,dx  \nonumber\\
& = 
-\int_{\Omega_1}\xi(x)\,\Delta \Psi_v(x)\,dx 
+ \dual{\dfrac{\partial \Psi_v}{\partial {\bf n}_{12}}}{\xi}{\gamma} \nonumber\\
&\hskip5cm 
+ \dual{v}{\xi}{\Gamma} \nonumber \\
& = 
- \dual{\Lambda_{\gamma}v}{\Lambda_{1}\Lambda_{2}\phi}{\gamma} \nonumber \\
& \hskip4cm
+ \dual{v}{\xi}{\Gamma} \nonumber.
\end{align}
Here, and in the sequel, in the duality bracket $\dual{v}{\xi}{\Gamma}$ one should interpret $\xi$ as being the trace of $\xi \in H^1(\Omega_{1})$ on $\Gamma$. 
We thus deduce that for all $\phi \in H^{1/2}_{m}(\gamma)$ and $v \in H^{-1/2}_{m}(\Gamma)$ we have
\begin{equation}\label{eq:dualvphi}
\dual{\Lambda_{\gamma} v}{\phi + \Lambda_1\Lambda_2 \phi }{\gamma}
= \dual{v}{\xi}{\Gamma} .
\end{equation}
Recall that for any given ${\widetilde \phi} \in H^{1/2}_{m}(\gamma)$, by Proposition \ref{prop:T12bij} there exists a unique $\phi \in H^{1/2}_{m}(\gamma)$ such that ${\widetilde \phi} = \phi +\Lambda_1\Lambda_2(\phi)$.
Setting $\phi := (I + \Lambda_{1}\Lambda_{2})^{-1}\widetilde{\phi}$ and then
$$\zeta := \zeta(\phi) = 
\zeta\left((I + \Lambda_{1}\Lambda_{2})^{-1}\widetilde{\phi}\right)
\qquad\mbox{and}\qquad 
\xi := \xi(-\Lambda_{2}(\phi)) \,,$$ 
we deduce from \eqref{eq:dualvphi}  that for any $v\in H^{-1/2}_{m}(\Gamma)$ and any $\widetilde{\phi} \in H^{1/2}_{m}(\gamma)$, by the very definition of $\Lambda_{\gamma}^* : H^{1/2}_{m}(\gamma) \into H^{1/2}_{m}(\Gamma)$, we have 
\begin{equation}\label{eq:dualvphi2}
\dual{v}{\Lambda_{\gamma}^*{\widetilde \phi}}{\Gamma} =  \dual{\Lambda_\gamma(v)}{{\widetilde \phi}}{\gamma} = \dual{v}{\xi}{\Gamma}. 
\end{equation}
Now to conclude the first step of our proof, assume that ${\widetilde \phi}\in H^{1/2}_m(\gamma)$ is such that $\Lambda_{\gamma}^*{\widetilde \phi} = 0$. Then the above identity \eqref{eq:dualvphi2} implies that for all $v \in H^{-1/2}_{m}(\Gamma)$ we have 
$\dual{v}{\xi}{\Gamma} = 0$.
This means that $\xi \equiv 0$ on $\Gamma$, which, thanks to the unique continuation property for the Laplace operator (see e.g. L. H\"ormander  \cite{HORM90}, theorem 8.6.5) and the fact that $\xi$ satisfies also condition \eqref{eq:dernormtheta1surGamma}, implies that $\xi\equiv 0$ in $\Omega_{1}$. This in turn implies that $\Lambda_{2}(\phi) \equiv 0$ on $\gamma$. However, thanks to \eqref{eq:Rel-Op-Lambda-2} and the fact that by our assumption on $\gamma$ the set $\Omega_{2}$ is a connected open domain, we conclude that $\zeta$ is constant in $\Omega_{2}$. Finally, $\phi$ is a constant on $\gamma$ and, since it has zero mean value there, we infer that $\phi \equiv  0$ on $\gamma$ and thus ${\widetilde \phi} \equiv 0$, that is $N(\Lambda_{\gamma}^*) = \{0\}$ and $R(\Lambda_{\gamma})$ is dense in $H_{m}^{-1/2}(\gamma)$.
\medskip

\noindent {\bf Step 2.} In this step we assume that the part of the boundary $\Gamma \subset \partial\Omega$, on which the Lagrangian control is applied, is as in figure 2. More precisely, we extend the domain $\Omega$ into a strictly larger domain ${\widetilde \Omega}$ in such a way that some relatively open part $\Gamma_{0}$ of $\Gamma$ lies in ${\widetilde \Omega}$. Now we consider a small ball $\omega \subset \subset {\widetilde \Omega} \setminus {\overline \Omega}$, and we set ${\widetilde \Omega}_{0} := {\widetilde \Omega} \setminus {\overline \omega}$ and ${\widetilde \Gamma}_{0} := \partial\omega$.

\begin{figure}[ht]
\begin{center}
\resizebox{!}{7cm}{\includegraphics{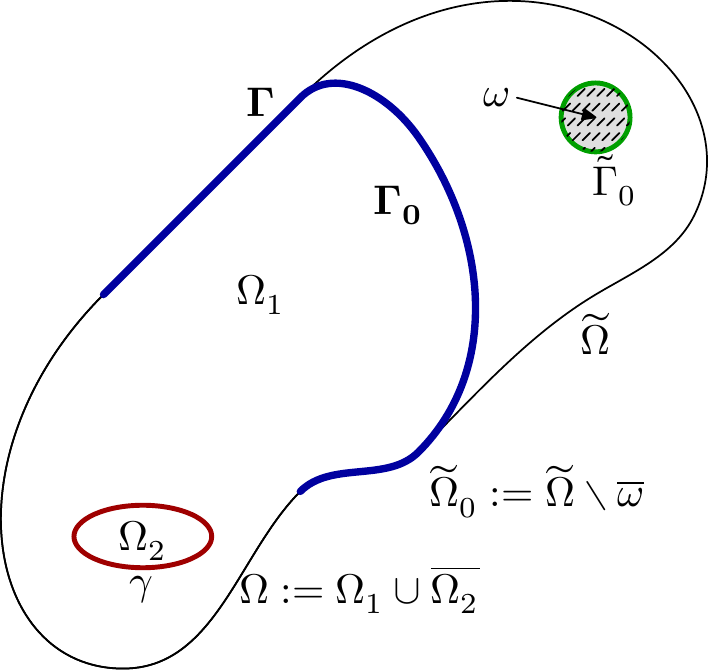}}
\end{center}
\caption{The case of the second step.}
\end{figure}

In the domain ${\widetilde\Omega}_{0}$ we may apply the result of the above Step~1: 
denote by $\Lambda_{0,\gamma}$ the mapping defined by
\begin{eqnarray*}
 \Lambda_{0,\gamma} : H^{-1/2}_m({\widetilde\Gamma}_{0}) & \into & H^{-1/2}_m(\gamma)\\
v_{0} & \longmapsto & \Lambda_{0,\gamma}(v_{0}):=\nabla\Psi_{0,v_{0}}\cdot{\bf n}_{12|\gamma}\nonumber
\end{eqnarray*}
where $\Psi_{0,v_{0}}$ has a mean value equal to zero on $\Omega$, and sastisfies 
$$-\Delta\Psi_{0,v_{0}} = 0 \quad 
\mbox{in }\; {\widetilde\Omega}_{0},
\qquad 
\nabla \Psi_{0,v_{0}} \cdot {\bf n} = v_{0}\; \mbox{on }\, {\widetilde\Gamma}_{0},
\quad \mbox{and }\; 
\nabla \Psi_{0,v_{0}} \cdot {\bf n} = 0\; \mbox{on }\, \partial{\widetilde\Omega},
$$
Then, according to what we have proved in Step 1, we know that $R(\Lambda_{0,\gamma})$ is dense in $H^{-1/2}_{m}(\gamma)$. Now we point out that if we set $v:=\nabla\Psi_{v}\cdot {\bf n}$ on $\partial\Omega$, and $\Psi_{v}:= (\Psi_{0,v_{0}})_{|\Omega}$ the restriction of $\Psi_{0,v_{0}}$ to $\Omega$, then $\Omega$ being smooth, we have $\Psi_{v}\in H^1(\Omega)$ and $-\Delta \Psi_{v} = 0$ in $\Omega$. Thus the mapping $v \mapsto \nabla\Psi_{v}\cdot {\bf n}_{12}$ corresponds to the mapping $v \mapsto \Lambda_{\gamma}(v)$, and we see that $R(\Lambda_{0,\gamma}) \subset R(\Lambda_{\gamma})$ (note that $\Psi_{v}$ being harmonic in $\Omega$, this implies that $v \in H^{-1/2}_{m}(\Gamma_{0})$): from this we infer that $R(\Lambda_{\gamma})$ is dense in $H^{-1/2}_{m}(\gamma)$, and the proof of Theorem \ref{lem:Im-Dense} is complete.
\end{proof}

\begin{remark}\label{rem:2-3}
Since for any $v\in H^{-1/2}_m(\Gamma)$, $\Psi_v$ defined by \eqref{eq:Psivharm} is harmonic in $\Omega$, we have that $\Lambda_{\gamma}(v)$ is as smooth as the manifold $\gamma$ itself. Thus $\Lambda_{\gamma}$ cannot be surjective. 
This means that in general it is not possible to have an exact Lagrangian controllability. However, in order to decsribe the process which allows us to deduce the approximate Lagrangian controllability from the density result of Theorem \ref{lem:Im-Dense}, we refer to \cite[section 2.2]{GlHo10}.
\qed
\end{remark}
%%%%%%%%%%%%%%%%%%%%%%%

%%%%%%%%%%%%%%%%%%%%%%%
\section{Specificity of the dimension 2}\label{sec:Runge}

\noindent Before dealing with an interpretation of the instabilities inherent to the problem under investigation, in this section we present some specific comments on the dimension 2. We refer to the paper \cite{GlHo08} by O. Glass and Th.\ Horsin for a thorough presentation of this approach. 

Let $\Omega \subset {\Bbb R}^2$ a domain, and let $f:\Omega\to \mathbb{C}$ be a complex valued function. We write $f = f_1 + {\rm i} f_2$ where $f_1$ and $f_2$ are real valued functions.
First we recall that $f$ satisfies the Cauchy-Riemann equations, or equivalently $f$ is holomorphic in $\Omega$ (see for instance W.~Rudin \cite{Rudin}), if and only if the vector 
$$V_{f}:=\begin{pmatrix}f_1\\ -f_2 \end{pmatrix}$$ 
satisfies the two conditions ${\rm div}(V_{f}) = \curl(V_{f}) = 0$.

Now, if $f$ satisfies the Cauchy-Riemann equation in $\Omega$, since $\curl (V_{f}) = 0$ in $\Omega$, then the vector valued function $V_{f} : \Omega \into {\Bbb R}^2$ is the gradient of some function $\Phi$ in $\Omega$, and thus finally, as $V_{f}$ satisfies ${\rm div}(V_{f}) = 0$ in $\Omega$,  then we conclude that $\Phi$ is a harmonic function defined in $\Omega$ and $V_{f} = \nabla\Phi$. The aim of this section is to give a constructive approximation of $f$ through the classical Runge's theorem, and thus to obtain an approximation procedure for $\nabla\Phi$.

With the notations introduced in Section \ref{sec:Intro} for Theorem \ref{th:lagrcontrpot}, in this section we will moreover assume that $T = 1$, and that the curve $\gamma_0$, as well as the maps $x \mapsto X(t,x)$ for each $t\in [0,1]$ are smooth, more precisely we assume that $\gamma_0\in C^\omega(\mathbb{S}^1,\mathbb{C})$ and $X\in C^\infty_0([0,T],C^\omega(\Omega)\cap C^\infty_0(\overline{\Omega}))$.

Throughout this section we will denote $\gamma_t:=\Phi^X(0,t,\gamma_0)$, that is the image of $\gamma_{0}$ under the flow of the vector field $X$. 

In this situation the proof of Theorem \ref{th:lagrcontrpot} relies on a compactness argument in time and the use of an appropriate version of the Cauchy-Kowalevsky's theorem (see \cite[Theorem 5.7.1']{Mo08}   ) for a precise statement) on one hand, and the Runge's approximation theorem, on the other hand. 

More precisely, it is shown that, for some integer $m\in \mathbb{N}^*$, there exists a finite sequence of times $t_0:=0< t_1 < ...< t_m< t_{m+1}= T =1$, and $m$ functions $\rho_i\in C^\infty_0(]t_{i-1},t_{i+1}[,[0,1])$ and $m$ functions $\phi_i$ harmonic on $\Omega$, satisfying conditions \eqref{eq:lagrcontpot.2} and \eqref{eq:lagrcontpot.3} at $t=t_i$, as well as \eqref{eq:deriveenormaldephiprochedeX}, such that
\begin{equation}
\phi(t,x):=\sum_{i=1}^m\rho_i(t)\phi_i(x),\label{eq:formdupot}
\end{equation} 
satisfies \eqref{eq:flotdansOmegabis} and \eqref{eq:controlagpotok}.

As the above definition \eqref{eq:formdupot} suggests, in order to prove the approximate lagrangian controllability, an option is to approximate the functions $\rho_i$ in time and the functions $\phi_i$ in $x$.  

\subsection{Constructing a Runge's approximation}\label{subsec:Runge}

We wish to find an explicit approximation procedure in the following Runge's approximation theorem (see \cite{Rudin} in particular for other remarkable properties deduced from this theorem).

\begin{theorem}Let $\Omega$ be an open subset of $\mathbb{C}$. Let $K$ a compact subset of $\Omega$ and $S \subset {\Bbb C}$ a set which has exactly one point in each connected component of $\mathbb{C}\setminus \Omega$, and $f:\Omega \to \mathbb{C}$ a holomorphic function. Then given any $\eps>0$, there exists a rational function $R$ whose poles are exactly the points of $S$, and moreover $R$ satisfies 
$$||f-R||_{L^\infty(K)}\leq \eps.$$
\end{theorem}

Though the proof can be given in a more general settings, in the sequel, for the sake of simplicity and clarity, we will assume that $\Omega$ is connected, and that $\mathbb{C}\setminus \Omega$ has one connected component. 

Let $\mathcal{O} \subset \Omega$ be an open set such that, for some integer $p \geq 0$, the set $\mathbb{C}\setminus \mathcal{O}$ has exactly $(p+1)$ connected components, each of them containing exactly one element of $S$, and verifying 
\begin{equation}\label{eq:defO}
\overline{\mathcal{O}}\subset \Omega,\qquad\mbox{and}\qquad
\forall t\in [0,1],\qquad \Phi^X(0,t,\gamma_0)\subset \mathcal{O}.
\end{equation}
We assume moreover that $X\in C_0^\infty([0,1],C^\infty_0(\Omega)\cap C^\omega(\mathcal{O}))$, and that there exists $\Phi\in C_0^\infty([0,1],C^\omega(\mathcal{O}))$ such that (recall that $X$ is divergence free)
$$\forall t\in [0,1],\qquad \nabla \Phi(t,\cdot)=X(t,\cdot)\; \mbox{ in }\,\mathcal{O}.$$
Let us consider a curve $\widetilde{\gamma}_0$ such that $\gamma_0\subset \inside(\widetilde{\gamma}_0)\subset \mathcal{O}$. For $t\in [0,1]$ if we denote by 
$$\widetilde{\gamma}_t:=\Phi^X(0,t,\widetilde{\gamma}_0),\label{eq:gammatilde}$$ 
then we clearly have
$$\Phi^X(0,t,\gamma_0)=\gamma_t\subset \inside(\widetilde{\gamma}_t).$$
We now define $f:[0,1]\times \Omega\to \mathbb{C}$ by the formula 
\begin{equation}V_{f(t,\cdot)}=\nabla \Phi(t,\cdot),\label{eq:defoff}\end{equation} 
and thus $f(t,\cdot)$ is  holomorphic on a neighborhood of $\gamma_t$. 
We may prove now the following:
\begin{theorem}\label{th:varyingrunge} For any $\eps>0$,
there exists a function $R \in C^\infty([0,1],C^\omega(\mathbb{C}\setminus S))$ such that for any $t\in [0,1]$, the function $z\mapsto R(t,z)$ is a rational function whose poles are exactly the points of $S$ and such that
\begin{equation}\label{eq:Runge-approx}
\sup_{t\in [0,1]}||f(t,\cdot) - R(t,\cdot)||_{L^\infty(\inside(\gamma_t))}\leq \eps.
\end{equation}
\end{theorem}

\begin{proof} For the sake of simplicity, we give the proof only in the case when $\mathbb{C} \setminus \Omega$ has one connected component. The reader will be convinced that through easy modifications the proof can be carried out in the general case. 

Let us choose $K$, a compact subset of $\Omega$, such that 
$$\overline{\bigcup_{t\in [0,1]}\inside(\gamma_t)}\subset \mbox{Int}(K).$$ 
Thanks to the compactness of $[0,1]$ and the continuity of $f$, for a given $\eps > 0$, there exist an integer $n \geq 1$, a positive number $\kappa$, and a sequence $0<  t_1 < \cdots < t_n <1$ such that $[0,1]=\cup_{j=0}^n((t_j-\kappa,t_j+\kappa)\cap [0,1])$ and 
$$\forall\, t \in\, (t_j-\kappa,t_j+\kappa)\cap [0,1],\qquad \sup_{z\in K}|f(t,z)-f(t_j,z)|\leq {\eps \over 2}.$$
Choose $(\phi_j)_{1\leq j \leq n}$ a partition of unity such that $\mbox{supp}(\phi_j)\subset (t_j-\kappa,t_j+\kappa)\cap [0,1]$, and 
also for $1 \leq j \leq n$ denote  
$$K_j := \overline{\bigcup_{t\in (t_j-\kappa,t_j+\kappa)\cap [0,1]}\inside(\gamma_t)}.$$
Now, thanks to Runge's theorem there exists $R_j$ a rational function whose poles are exactly the points of $S$ such that 
$$\sup_{z\in K_j}|f(t_j,z)-R_j(z)|\leq {\eps \over 2}.$$
At this point it is clear that if we set 
$$R(t,z) := \sum_{j=1}^n\phi_j(t)R_j(z),$$
then by construction $R$ satisfies \eqref{eq:Runge-approx}.
\end{proof}

\begin{remark}
As a matter of fact, it is possible to give an explicit construction of $R$. Indeed, first, the partition of unity $(\phi_j)_{1\leq j \leq n}$ can be constructed by means of the well-known function 
$$x\mapsto \Psi(x)\Psi(1-x)$$ where 
$$\Psi(x)=\int_{-\infty}^x\psi(t)dt,$$ with 
$$\psi(x)=\begin{cases}0\quad \mbox{ if }x<0,\cr e^{-1/x}\quad\mbox{ otherwise}.\end{cases}$$
Next, with our assumptions on $\Omega$, we can give an explicit function $R_j$. Indeed such an explicit construction is given, for example in \cite{sarasond}.\qed
\end{remark}

\subsection{Application to the controllability problem} 

We now explain how we apply this approximation to the Lagrangian controllability by means of harmonic flows and in particular how we deal with condition \eqref{eq:euler.4}.
For simplicity, we will assume that $\mathbb{C}\setminus \Omega$ has only one connected component. 
%We consider $f$ as in the preceding Subsection \ref{subsec:Runge}, and denote by $R$ a Runge's approximation given by theorem \ref{th:varyingrunge}.

Consider a simply connected open neighborhood $\mathcal{U}$ of $\mathcal{O}$ (recall that $\mathcal O$ is defined by \eqref{eq:defO}) such that
$\overline{\mathcal{O}}\subset \mathcal{U}\subset \overline{\mathcal{U}}\subset \Omega$, and denote by $\mathcal{V}$ a simply connected neighborhood of $\partial \Omega\setminus \Gamma$ such that $\mathcal{V}\cap \mathcal{U}=\emptyset$.

Let us recall the Mergelyan's theorem (see \cite{Rudin})
\begin{theorem}\label{th:mergelyan}Let $O$ be a relatively compact open set of $\mathbb{C}$ such that $\mathbb{C}\setminus O$ is connected and consider $h$ a continuous map defined in $\overline{O}$, holomorphic in $O$. Then for any $\eps>0$ there exists a polynomial $P$ such that
$$\forall z\in \overline{O},\quad |P(z)-h(z)|\leq \eps.$$
\end{theorem}

\noindent Now let $f$ be given as in \eqref{eq:defoff} and, for a given $\eps$, let $R$ be the Runge's approximation given by Theorem \ref{th:varyingrunge}.
{
If $h$ is a continuous function defined in $\overline{\Omega}$ such that $h \equiv 1$ on $\mathcal{U}$ and $h \equiv 0$ on $\mathcal{V}$, then using the above Mergelyan's theorem with $O=\mathcal{U}\cup \mathcal{V}$ on which $h$ is holomorphic, there exists a polynomial $P_\eps$ such that 
$$\forall z\in \mathcal{U},\qquad |P_\eps(z)-1|\leq \eps/ \|R\|_{W^{1,\infty}(\mathcal{U)}},$$
and such that 
$$\forall z\in \mathcal{V}, \qquad |P_\eps(z)|\leq \eps/\|R\|_{W^{1,\infty}(\mathcal{V})}.$$

\begin{figure}[!ht]
\begin{center}
\resizebox{!}{5cm}{\includegraphics{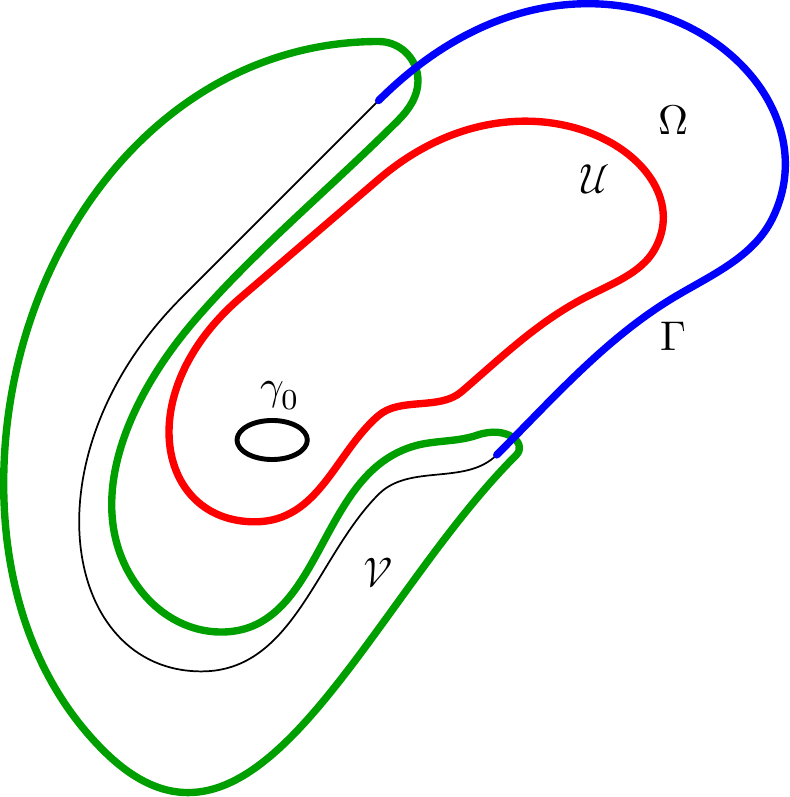}}
\end{center}
\caption{Illustration of the situation in section 3.2.}
\end{figure}

Precisely we have to impose that $\mathbb{C}\setminus O$ is connected. But in fact due to our construction we can assume that $\Gamma$ and $\partial \Omega\setminus \overline{\Gamma}$ are connected (it suffices not to control on the other part of $\Gamma$), and thus we can take $\mathcal U$ connected and simply connected and since $\Omega$ is supposed to be simply connected, we can take $\mathcal V$ connected and simply connected as well.
 
Thus we get 
\begin{align}
\sup_{t\in [0,1]}\|P_\eps R(t,\cdot) -f\|_{L^\infty(\inside({\gamma}_t))} & \leq \sup_{t\in [0,1]}\|P_\eps R(t,\cdot)-R(t,\cdot)\|_{L^\infty(\inside({\gamma}_t))} + \cr
& \qquad\qquad \sup_{t\in [0,1]}\|R(t,\cdot)-f(t,\cdot)\|_{L^\infty(\inside({\gamma}_t))} \cr
&\leq 2\eps\label{eq:estimapprox2d.1}
\end{align} 
and naturally
\begin{equation}\|P_\eps R\|_{L^\infty(\mathcal{V})}\leq \eps.\label{eq:estimapprox2d.2}
\end{equation}
Let us remark that, if we consider an intermediate smooth Jordan curve $\widehat{\gamma}_0$ such that $\gamma_0\subset {\inside(\widehat{\gamma}_0)}\subset \widehat{\gamma}_0\subset {\inside(\widetilde{\gamma}_0)}$, and if we denote 
$\widehat{\gamma}_t:=\Phi^X(0,t,\widehat{\gamma}_0)$, then the preceding approximation can be done by replacing $\gamma_0$ by $\widehat{\gamma}_0$. 
Proceeding analogously, we get the same estimate as  \eqref{eq:estimapprox2d.1}  where $\gamma_t$  is replaced by $\widehat{\gamma}_t$. 

However, since $P_\eps R$ and $f$ are holomorphic and thus harmonic where defined, by standard elliptic estimates, for some constant $C > 0$ we have
\begin{equation}
\sup_{t\in [0,1]}\|P_\eps R(t,\cdot)-f\|_{W^{1,\infty}(\inside(\gamma_t))}\leq C\eps,
\end{equation}
and, by choosing $\widetilde{\mathcal{V}}$ a neighborhood of $\partial \Omega \setminus \Gamma$ such that 
$\overline{\widetilde{\mathcal{V}}}\subset {\mathcal{V}}$, we have (again using elliptic estimates)

\begin{equation}
 \sup_{t\in [0,1]}\|P_\eps R\|_{W^{1,\infty}(\widetilde{\mathcal{V}})}\leq C\eps, \label{eq:fepsaubord}
\end{equation}
where $C$ depends only on $d(\widetilde{\mathcal{V}},\partial \mathcal{V})$ and $\min_{t\in [0,1]}d(\gamma_t,\widehat{\gamma}_t)$ (let us remark that since $\gamma_0\subset \mbox{Int}(\inside(\widehat{\gamma_0}))$ then $\gamma_t\subset \mbox{Int}(\inside(\widehat{\gamma_t}))$ since $\phi^X$ is the flow of $X$, and since by uniqueness of the solution of an ordinary differential equation, a compactness argument implies that we have $\min_{t\in [0,1]}d(\gamma_t,\widehat{\gamma}_t)=\inf_{t\in [0,1]}d(\gamma_t,\widehat{\gamma}_t)>0$).

We denote $f_\eps(t,z):=P_\eps(z) R(t,z)$. Since $\Omega$ is connected, for each $t\in [0,1]$, the vector valued function $V_{f_\eps}$ is the gradient of some harmonic function $\psi_\eps$, but it does not necessarily satisfy \eqref{eq:lagrcontpot.3}.
In order to construct a function which satisfies this condition, let us consider a function $k$ on $[0,1]\times \partial \Omega$ such that $\forall t\in [0,1]$
\begin{align}
&k(t,\cdot) = V_{f_{\eps}}\quad \mbox{on }\, \partial \Omega\setminus \Gamma\cr
&\|k(t,\cdot)\|_{C^{2}(\partial \Omega)}\leq C \| V_{f_\eps}(t,\cdot)\|_{C^2(\partial \Omega \setminus \Gamma)}\cr
&\int_{\partial \Omega}k(t)d\sigma=0,
\end{align}
for some constant $C$ independend of $\eps$ (such a $k$ can be constructed using Urysohn's extension theorem).
For any $t\in [0,1]$, let us now consider $\zeta(t,\cdot)$ harmonic in $\Omega$ such that 
$$\dfrac{\partial \zeta}{\partial {\bf n}}(t,\cdot)=k(t,\cdot) \qquad \mbox{on }\, \partial \Omega, \qquad \mbox{and}\qquad 
\int_{\partial \Omega}\zeta(t,\cdot)d\sigma =0.$$
Thanks to standard elliptic estimates we have, for some constants $C >0$
\begin{equation}\label{eq:estofzeta}
\sup_{t\in [0,1]}\|\zeta(t,.)\|_{C^2(\overline{\Omega})} \leq C\, \|k(t,.)\|_{C^2(\partial \Omega)}\leq C\, \eps.
\end{equation}
Finally, consider a function $\rho\in C^\infty([0,1])$ such that
$$\forall\, t\in [0,1],\quad \rho(t)\in [0,1],\qquad
\rho(0) = \rho(1) =0,$$
and define the new function $\widehat{f}_\eps(t,z) := \rho(t)(f_\eps(t,z) - \zeta(t,z))$, the vector valued function $V_{\widehat{f}_\eps}$ is the gradient of a function that satisfies the conditions of theorem \ref{th:lagrcontrpot} provided that $\rho=1$ on some $[\eta,1-\eta]$ with $\eta>0$ sufficiently small. 

To finish with this section, let us remark that in order to achieve our approximation argument, one has to explain how we can explicitely construct $P_\eps$ and how one can approximate $\zeta$.
For the former, if one closely looks at the proof of the Mergelyan's theorem given in \cite{Rudin}, one sees that it suffices to give an explicit construction of the Runge's approximation which is done in \cite{sarasond}.
For the latter, it suffices to apply a finite element method to approximate $\Upsilon_\eps$.

%%%%%%%%%%%%%%%%%%%%%%%

%%%%%%%%%%%%%%%%%%%%%%%
\section{A precise analysis of ill-posedness}\label{sec:Ill-posedness}

\noindent The analysis undertaken in the previous sections shows the possibility of an approximate Lagrangian control which, in general, cannot be exact. As a matter of fact, this can be interpreted as an issue related to the ill-posedness of the problem consisting in the determination of a harmonic function in a domain $\Omega$ with Cauchy data on some part of its boundary $\partial\Omega$.

Indeed, consider the following problem, which is a simplified version of the Lagrangian control under study in this paper.
Let $\Omega$ be the rectangular domain 
$$\Omega := (0,\pi )\times (0,\ell) \subset
{\Bbb R}^2 \quad\mbox{for some }\, \ell >0, $$
and denote by $\Gamma_{0}$, $\Gamma_{1}$ and $\Gamma$ the following parts of the boundary:
\begin{equation}
\Gamma_0:=[0,\pi ] \times \left\{0 \right\},\qquad \Gamma_1 := [0,\pi ]\times \left\{\ell
\right\},\qquad \Gamma := \partial \Omega \setminus (\Gamma_0 \cup \Gamma_1).
\end{equation}
Moreover, for a given $\ell_{*}$ such that $0 < \ell_{*} < \ell$, consider 
\begin{equation}\label{eq:Gamma-*}
\Gamma_{*} := [0,\pi]\times \left\{\ell_{*}\right\}.
\end{equation} 
The problem we want to analyze is this: for a given $g_{*} \in H^{-1/2}(\Gamma_{*})$ find a Neumann boundary data $g_{0} \in H^{-1/2}(\Gamma_{0})$ such that there exists a harmonic function $u \in H^1(\Omega)$ such that
\begin{equation}\label{eq:Pb-g0-g*} % eq:Pb-g0-g1
-\Delta u = 0\; \mbox{in }\, \Omega, \qquad\mbox{and}\qquad
{\partial u \over\partial{\bf n} } = g_{0}\;\mbox{on }\, \Gamma_{0}\quad \mbox{and }\,
{\partial u \over\partial{\bf n} } = g_{*}\;\mbox{on }\, \Gamma_{*}.
\end{equation}
Note that this problem is similar to the one considered in Theorem \ref{lem:Im-Dense}, but here the target curve $\gamma$ is a simple line which intersects the boundary of $\Omega$. In the limit case when $\ell_{*} = \ell$, the problem would be
\begin{equation}\label{eq:Pb-g0-g1} 
-\Delta u = 0\; \mbox{in }\, \Omega, \qquad\mbox{and}\qquad
{\partial u \over\partial{\bf n} } = g_{0}\;\mbox{on }\, \Gamma_{0}\quad \mbox{and }\,
{\partial u \over\partial{\bf n} } = g_{1}\;\mbox{on }\, \Gamma_{1},
\end{equation}
for which one sees obviously that any $g_{0}\in H^{-1/2}(\Gamma_{0})$ such that 
$$\langle g_{0},1\rangle_{H^{-1/2}(\Gamma_{0}),H^{1/2}(\Gamma_{0})} + 
\langle g_{1},1\rangle_{H^{-1/2}(\Gamma_{1}),H^{1/2}(\Gamma_{1})} = 0$$
yields a function $u\in H^1(\Omega)$ solution to the above equation \eqref{eq:Pb-g0-g1}.
\medskip

To find a $g_{0}$ such that \eqref{eq:Pb-g0-g*} is satisfied, or rather such that the normal derivative $\partial u/\partial{\bf n}$ on $\Gamma_{*}$ is an approximation of $g_{*}$, we first solve an auxiliary boundary value problem, namely for a fixed $f_{0} \in H^{1/2}(\Gamma_{0})$ and $g_{1} \in H^{-1/2}(\Gamma_{1})$ we seek $v \in H^1(\Omega)$ solution to
\begin{equation}\label{eq:Aux-f0-g1}
\begin{cases}
-\Delta v = 0 & \mbox{in }\, \Omega \cr
v = f_0 & \mbox{on }\, \Gamma_0 \cr
\displaystyle {\partial v \over \partial {\bf n} } = g_{1} & \vphantom{{\partial v \over \partial {\bf n} } }\mbox{on }\, \Gamma_{1} \cr
\displaystyle {\partial v \over \partial {\bf n} } = 0 & \vphantom{{\partial v \over \partial {\bf n} } }\mbox{on }\, \Gamma.\cr
\end{cases}
\end{equation}
For later comments, we attract the reader's attention to the fact that at this point we have added a Neumann boundary condition on the part $\Gamma$ of the boundary, where problems \eqref{eq:Pb-g0-g*} or \eqref{eq:Pb-g0-g1}
 do not impose any such restriction. Actually one may consider other boundary conditions on $\Gamma$, but we shall develop this aspect later.

Clearly, for any given pair $(f_{0},g_{1}) \in H^{1/2}(\Gamma_{0}) \times H^{-1/2}(\Gamma_{1})$, equation \eqref{eq:Aux-f0-g1} has a unique solution $v \in H^1(\Omega)$, and if we can find a $(f_{0},g_{1}) \in H^{1/2}(\Gamma_{0}) \times H^{-1/2}(\Gamma_{1})$ such that this solution $v$ satisfies 
$${\partial v \over\partial {\bf n} } = g_{*}\quad
\mbox{on }\,\Gamma_{*}, \qquad \mbox{or }\, \left\|{\partial v \over\partial{\bf n} } - g_{*}\right\|_{H^{-1/2}(\Gamma_{*})} \leq \epsilon,$$
for a certain $\epsilon >0$, then $v$ solves problem \eqref{eq:Pb-g0-g*}, or an approximation of it.
It is also clear that not all $g_{*} \in H^{-1/2}(\Gamma_{*})$ may be attained, and thus again the only hope is to approximate $g_{*}$ with $\partial v/\partial{\bf n}$.
\medskip

We shall describe the solution of \eqref{eq:Aux-f0-g1} in terms of eigenfunctions of two Steklov eigenvalue problems associated to this boundary value problem, and then we shall give necessary and sufficient conditions on the pair $(f_{0},g_{0})$ ensuring the existence and uniqueness of a solution $u\in H^1(\Omega)$ of
\begin{equation}\label{eq:Hadamard}
\begin{cases}
-\Delta u = 0 & \mbox{in }\, \Omega \cr
u = f_0 & \mbox{on }\, \Gamma_0 \cr
\displaystyle {\partial u \over \partial {\bf n} } = g_0 & \vphantom{{\partial u \over \partial {\bf n} } }\mbox{on }\, \Gamma_0\cr
\displaystyle {\partial u \over \partial {\bf n} } = 0 & \vphantom{{\partial u \over \partial {\bf n} } }\mbox{on }\, \Gamma.\cr
\end{cases}
\end{equation}
For $j = 0$ or $j = 1$, the Steklov eigenfunctions $\psi_{j,k}$ are defined as solutions to
\begin{equation}\label{eq:Steklov-j}
\begin{cases}
-\Delta \psi_{j,k} = 0 & \mbox{in }\, \Omega \cr
\displaystyle {\partial \psi_{j,k} \over \partial {\bf n} } = \mu_{j,k}\psi_{j,k} & \vphantom{{\partial u \over \partial {\bf n} } }\mbox{on }\, \Gamma_j\cr
(1 - j)\displaystyle {\partial \psi_{j,k} \over \partial {\bf n} }
+ j \psi_{j,k} = 0 & 
\vphantom{{\partial v \over \partial {\bf n} } }
\mbox{on }\, \Gamma_{1-j} \cr
\displaystyle {\partial \psi_{j,k} \over \partial {\bf n} } = 0 & \vphantom{{\partial v \over \partial {\bf n} } }\mbox{on }\, \Gamma \cr
\end{cases}
\end{equation}
with the normalization
$\int_{\Gamma_{j}}\psi_{j,k}(\sigma)\psi_{j,k'}(\sigma)\,d\sigma = \delta_{k'k}$, so that the Steklov eigenfunctions $(\psi_{j,k})_{k \geq 0}$ form a Hilbert basis of $L^2(\Gamma_{j})$.
One checks easily that 
\begin{equation}
\psi_{0,0}(x,y) := {1 \over \sqrt{\pi}}, \qquad
\psi_{1,0}(x,y) := {1 \over \ell\sqrt{\pi}}\, y, 
\end{equation}
while for $k \geq 1$ we have
\begin{equation}\label{eq:psi-0k}
\psi_{0,k}(x,y) := \sqrt{{2 \over \pi}}\,\cos(kx){\cosh(k(\ell-y)) \over \cosh(k\ell)}\, ,
\end{equation}
and
\begin{equation}\label{eq:psi-1k}
\psi_{1,k}(x,y) := \sqrt{{2 \over \pi}}\,\cos(kx){\sinh(ky) \over \sinh(k\ell)}\, .
\end{equation}
The eigenvalues $\mu_{j,k}$ are given by $\mu_{0,0} := 0$, and $\mu_{1,0} := 1/\ell$, while
\begin{equation}\label{eq:Steklov-Eigen}
\mu_{0,k} := k\,\tanh(k\ell), \qquad
\mu_{1,k} := k\,\cotanh(k\ell)\qquad \mbox{for }\, k \geq 1.
\end{equation}
If $(f_{0},g_{1}) \in H^{1/2}(\Gamma_{0})\times H^{-1/2}(\Gamma_{1})$, we shall denote for $k \geq 0$ 
$$f_{0,k} := \int_{0}^\pi f_{0}(x)\psi_{0,k}(x,0)\,dx,\qquad
g_{1,k} := \int_{0}^\pi g_{1}(x)\psi_{1,k}(x,\ell)\,dx.
$$
Note that for the definition of $g_{1,k}$ we should have written duality brackets between $H^{-1/2}(\Gamma_{1})$ and $H^{1/2}(\Gamma_{1})$, instead of an integral over $(0,\pi)$, but clearly there is no risk of ambiguity. Also we have that
$$f_{0} \in H^{1/2}(\Gamma_{0}) \iff |f_{0,0}|^2 + \sum_{k \geq 1} \mu_{0,k}|f_{0,k}|^2 < \infty,$$
the right hand side being equivalent to the norm of $f_{0}$ in $H^{1/2}(\Gamma_{0})$. Analogously  
$$g_{1} \in H^{-1/2}(\Gamma_{1}) \iff |g_{1,0}|^2 + \sum_{k \geq 1} {1 \over \mu_{1,k}}|g_{1,k}|^2 < \infty,$$
again the right hand side being equivalent to the norm of $g_{1}$ in $H^{-1/2}(\Gamma_{1})$.

Now we can state the following auxiliary result:

\begin{lemma} 
Let $(f_{0},g_{1}) \in H^{1/2}(\Gamma_{0})\times H^{-1/2}(\Gamma_{1})$ be given.
With the above notations for the Steklov eigenvalues and eigenfunctions, and for the Fourier coefficients $f_{0,k}, g_{1,k}$, the solution of \eqref{eq:Aux-f0-g1} is given by
\begin{equation}\label{eq:Sol-v-series}
v = \ell\, g_{1,0}\,\psi_{1,0} +  
\sum_{k \geq 1} {\tanh(k\ell) \over k}\, g_{1,k}\,\psi_{1,k}
+ \sum_{k \geq 0} f_{0,k}\,\psi_{0,k} 
\end{equation}
and on $\Gamma_{0}$ we have:
\begin{equation}\label{eq:dv-dn}
{\partial v \over\partial{\bf n}} = {-1 \over \sqrt{\pi}} \, g_{1,0} + \sqrt{{2 \over \pi}} \,
\sum_{k \geq 1}\left[ k\,\tanh(k\ell)\, f_{0,k} - {1 \over \cosh(k\ell)} \, g_{1,k}\right] \cos(kx).
\end{equation}
\end{lemma}

\begin{proof}
We know that $(\psi_{0,k}(\cdot,0))_{k \geq 0}$ and $(\psi_{1,k}(\cdot,\ell))_{k \geq 0}$ are  Hilbert bases of $L^2(\Gamma_{0})$ and $L^2(\Gamma_{1})$ respectively (actually up to an appropriate normalization $(\psi_{1,k}(\cdot,\ell))_{k\geq 0}$ can be also considered as a Hilbert basis in $H^{-1/2}(\Gamma_{1})$). Now, if we express $f_{0}$ and $g_{1}$ in terms of their coefficients in these bases, since $v$ is entirely determined by its traces on $\Gamma_{0}$ and $\Gamma_{1}$ we may write  
$$v = \sum_{k \geq 0}\alpha_{k} \psi_{0,k} + \sum_{k \geq 0} \beta_{k}\psi_{1,k}.$$
In order to find the coefficients $\alpha_{k}$ and $\beta_{k}$, it is sufficient to multiply equation \eqref{eq:Aux-f0-g1} by $\psi_{0,k}$ and $\psi_{1,k}$, and one finds easily the expression given by \eqref{eq:Sol-v-series}. 

For the determination of the normal derivative $\partial v/\partial{\bf n}$ on $\Gamma_{0}$, one considers first a finite sum in \eqref{eq:Sol-v-series}, and then pass to the limit in $H^{-1/2}(\Gamma_{0})$, since clearly the series in \eqref{eq:dv-dn} converges in this space, thanks to the assumptions on $f_{0}$ and $g_{1}$.
\end{proof}

From the expression of $\partial v/\partial{\bf n}$ it is clear that we may infer the following:

\begin{corollary} Let $g_{1} \in H^{-1/2}(\Gamma_{1})$ be given. For any  $f_{0} \in H^{1/2}(\Gamma_{0})$ and $g_{0} \in H^{-1/2}(\Gamma_{0})$ the solution $v$ of \eqref{eq:Aux-f0-g1} is solution to 
\eqref{eq:Pb-g0-g*} % eq:Pb-g0-g1
if and only if 
we have $g_{0,0} = - g_{1,0}$ and for all $k \geq 1$ 
\begin{equation}\label{eq:CNS-f0-g0}
g_{0,k} = k\,\tanh(k\ell)\, f_{0,k} - {1 \over \cosh(k\ell)} \, g_{1,k}.
\end{equation}
\end{corollary}

One sees that, as we already pointed out, problem \eqref{eq:Pb-g0-g1} has infinitely many solutions, since for each given $f_{0} \in H^{1/2}(\Gamma_{0})$ one may determine $g_{0}$ thanks to the above corollary, in which case the solution of \eqref{eq:Pb-g0-g1} obtained in this way satisfies moreover $\partial u/\partial{\bf n} = 0$ on $\Gamma = \partial\Omega \setminus (\Gamma_{0}\cup \Gamma_{1})$.
\medskip

Incidently, the above analysis shows that in order to find a harmonic function $u$ with Cauchy data $(f_{0},g_{0})$ on $\Gamma_{0}$, more precisely in order to solve the following problem: find $u\in H^1(\Omega)$ such that
\begin{equation}\label{CL1}
\begin{cases}
-\Delta u = 0 & \mbox{in }\, \Omega \cr
u = f_0 & \mbox{on }\, \Gamma_0 \cr
\displaystyle {\partial u \over \partial {\bf n} } = g_0 & \vphantom{{\partial u \over \partial {\bf n} } }\mbox{on }\, \Gamma_0\cr
\displaystyle {\partial u \over \partial {\bf n} } = 0 & \vphantom{{\partial u \over \partial {\bf n} } }\mbox{on }\, \Gamma\,,\cr
\end{cases}
\end{equation}
one possibility is to find $g_{1} \in H^{-1/2}(\Gamma_{1})$ such that the solution $v$ of \eqref{eq:Aux-f0-g1} satisfies $\partial v/\partial{\bf n} = g_{0}$. Thus we may state the following necessary and sufficient condition on the compatibility of $f_{0},g_{0}$:

\begin{proposition} \label{Thm1} 
Let $f_0 \in H^{1/2}(\Gamma_0)$ and $g_0 \in H^{-1/2}(\Gamma_0)$. With the above notations, equation \eqref{CL1} has a unique solution $u \in H^1(\Omega )$ if, and only if, the Cauchy boundary data $f_0,g_0$ satisfy the following compatibility condition:
\begin{equation} \label{CLCNS}
\sum_{k=1}^{\infty }k \sinh^2(k\ell) \left(f_{0,k} - {\cotanh(k\ell)\over k }\, g_{0,k} \right)^2 <
\infty .
\end{equation}
Moreover, when the above condition is satisfied, the solution $u$ is given by
$$u = \sum_{k \geq 0}f_{0,k}\psi_{0,k} - \ell g_{0,0}\psi_{1,0} + 
\sum_{k \geq 1} {\sinh^2(k\ell) \over \cosh(k\ell) }
\left(f_{0,k} - {\cotanh(k\ell) \over k} g_{0,k}\right) \psi_{1,k}\, ,$$
and there exist two positive constants $c_1,c_2$ such that if we denote by  $\|(f_{0},g_{0})\|_{*}^2$ the quantity
$$|f_{0,0}|^2 + |g_{0,0}|^2 + 
\sum_{k=1}^{\infty }k|f_{0,k}|^2 + 
\sum_{k\geq 1} k\sinh^2(k\ell)
\left(f_{0,k} - {\cotanh(k\ell)\over k }\, g_{0,k} \right)^2\, ,$$
then we have
\begin{equation}\label{eq:EstimCL}
c_{1}\, \|(f_{0},g_{0})\|_{*}^2 
\leq  \|u\|_{H^1(\Omega)}^2 \leq 
c_{2}\, \|(f_{0},g_{0})\|_{*}^2 .
\end{equation}
\end{proposition}

\begin{remark} 
Note that in our analysis of the resolution of \eqref{eq:Hadamard} we began with the resolution of the mixed boundary value problem \eqref{eq:Aux-f0-g1}, while we could have proceeded with another choice, for instance by solving 
\begin{equation}
\begin{cases}
-\Delta v = 0 & \mbox{in }\, \Omega \cr
v = f_0 & \mbox{on }\, \Gamma_0 \cr
v = f_1 & \mbox{on }\, \Gamma_1\cr
\displaystyle {\partial v \over \partial {\bf n} } = 0 & \vphantom{{\partial v \over \partial {\bf n} } }\mbox{on }\, \Gamma.\cr
\end{cases}
\label{D1}
\end{equation}
for some $f_{1} \in H^{1/2}(\Gamma_{1})$. Then we would have found a condition on $f_{0},f_{1}$ such that for a given $g_{0}$ we have $\partial v/\partial{\bf n} = g_{0}$, where $v$ is the unique solution of the above equation \eqref{D1}. In this case the Steklov eigenvalues and eigenfunctions $\psi_{1,k}$ should be replaced with $\mu_{1,0} := 0$ and $\psi_{1,0} := 1/\sqrt{\pi}$, and for $k \geq 1$
$$
\mu_{1,k} :=  k\, \tanh(k\ell) = \mu_{0,k}, \qquad
\psi_{1,k}(x,y) :=  \sqrt{{2 \over \pi}}\,\cos(kx){\cosh(ky) \over \cosh(k\ell)}.$$
Then condition \eqref{CLCNS} would be replaced with the following necessary and sufficient condition
$$\sum_{k=1}^{\infty }k \cosh^2(k\ell) \left(f_{0,k} - {\tanh(k\ell)\over k }\, g_{0,k} \right)^2 <
\infty ,$$
which is equivalent to \eqref{CLCNS}. \qed
\end{remark}
\medskip
\begin{remark}
If $\Gamma_{*}$ is as in \eqref{eq:Gamma-*} and $g_{*} \in H^{-1/2}(\Gamma_{*})$, let $u \in H^1(\Omega)$ be the solution of \eqref{CL1} for a compatible pair $(f_{0},g_{0}) \in H^{1/2}(\Gamma_{0})\times H^{-1/2}(\Gamma_{0})$, and let us compute $\partial u/\partial y$ restricted to $\Gamma_{*}$ and compare it with $g_{*}$. We have, for $k \geq 1$
\begin{align*}
{\partial \psi_{0,k}(x,\ell_{*}) \over\partial y } &= {-k\sinh(k(\ell -\ell_{*})) \over \cosh(k\ell)}\, \sqrt{{2 \over \pi }} \cos(kx) \\
{\partial \psi_{1,k}(x,\ell_{*}) \over\partial y } &= {k\sinh(k\ell_{*}) \over \cosh(k\ell)}\, \sqrt{{2 \over \pi }} \cos(kx),
\end{align*}
while for $k= 0$ we have
$${\partial\psi_{0,0} \over \partial{\bf n}} = 0, \qquad {\partial\psi_{1,0} \over \partial{\bf n}} = {1 \over \ell\sqrt{\pi} }.$$
Setting for $k \geq 1$
\begin{align*}
\alpha_{k} &:= {-k\sinh(k(\ell -\ell_{*})) \over \cosh(k\ell)}f_{0,k} \\
\beta_{k} & := {\sinh^2(k\ell) \over \cosh^2(k\ell) }
\left(f_{0,k} - {\cotanh(k\ell) \over k} g_{0,k}\right) \, k\sinh(k\ell_{*}),  
\end{align*}
we have
$${\partial u(x,\ell_{*}) \over\partial{\bf n} } = {-g_{0,0} \over \sqrt{\pi}} + \sqrt{{2 \over \pi}} \sum_{k \geq 1} (\alpha_{k} + \beta_{k}) \cos(kx).$$
Therefore one sees that the necessary and sufficient condition for the existence of $(f_{0},g_{0})$ such that condition \eqref{CLCNS} is satisfied, and moreover we have $g_{*} = \partial u(\cdot,\ell_{*})/\partial{\bf n}$, is that the Fourier coefficients of $g_{*}$
$$a_{0} := {1 \over \sqrt{\pi}}\int_{0}^\pi g(x)\,dx, \qquad a_{k} := \sqrt{{2 \over \pi}}\int_{0}^\pi g_{*}(x)\,\cos(kx)\,dx\quad\mbox{for }\, k \geq 1$$
are such that $a_{k} = \alpha_{k} + \beta_{k}$ and $a_{0} = -g_{0,0}/\sqrt{\pi}$. Since $\alpha_{k}$ has an exponential decay (of order $\exp(-k\ell_{*})$) and $\beta_{k}$ also has an exponential decay of order $\exp(-k(\ell -\ell_{*}))$, due to the condition \eqref{CLCNS}, one sees that $a_{k}$ must have a decay of order 
$\exp(-k\max(\ell_{*},\ell-\ell_{*}))$,
and in particular one sees that $g_{*}$ must be analytic on $(0,\pi)$.

If one wishes only to approximate a given $g_{*} \in H^{-1/2}(\Gamma_{*})$, clearly one can do so by taking a finite sum
$${a_{0} \over \sqrt{\pi}} + \sqrt{{2 \over \pi}}\sum_{k=1}^n a_{k}\cos(kx),$$
and then find coefficients $f_{0,k},g_{0,k}$ for $1 \leq k \leq n$, such that
$$a_{k} = {-k\sinh(k(\ell -\ell_{*})) \over \cosh(k\ell)}f_{0,k} + {\sinh^2(k\ell) \over \cosh^2(k\ell) }
\left(f_{0,k} - {\cotanh(k\ell) \over k} g_{0,k}\right) \, k\sinh(k\ell_{*}).$$
For instance one choice may be
\begin{align*}
f_{0,k} &:= {-\cosh(k\ell) \over k\sinh(k(\ell - \ell_{*}))}a_{k}, \\
g_{0,k} &:= k\tanh(k\ell)f_{0,k} = {- \sinh(k\ell) \over \sinh(k(\ell - \ell_{*}))}a_{k},
\end{align*}
which shows why a numerical instability appears since, for instance, an error in the coefficient $a_{n}$ of order $\epsilon$ is transmitted as an error of order $\epsilon\exp(n\ell_{*})/n$ in the determination of $f_{0,n}$. One can check that any other choice of $f_{0,k},g_{0,k}$ yields the same type of numerical error. \qed
\end{remark}
\medskip

In the following corollary we state a noteworthy result for the case in which the domain $(0,\pi)\times (0,\ell)$ is replaced by $(0,\ell_{1})\times (0,\ell_{2})$: its proof is straightforward after a slight adaptation of the Steklov eigenfunctions $\psi_{j,k}$.

\begin{corollary} \label{thm:Cas-gene}
Let $\Omega := (0,\ell_{1})\times(0,\ell_{2})$ for some $\ell_{1} > 0$ and $\ell_{2} > 0$, and denote
$$\Gamma_{0}:=[0,\ell_{1}]\times\{0\},\quad 
\Gamma_{1} := [0,\ell_{1}]\times\{\ell_{2}\},\quad
\Gamma := \partial\Omega \setminus\left(\Gamma_{0}\cup\Gamma_{1}\right).$$
Let the Steklov eigenfunctions and eigenvalues $\psi_{j,k},\mu_{j,k}$ be defined by $\psi_{0,0}:= 1/\sqrt{\ell_{1}}$, with $\mu_{0,0} := 0$, and $\psi_{1,0} := y/(\ell_{2}\sqrt{\ell_{1}})$ with $\mu_{1,0} = 1/\ell_{2}$, while for $k \geq 1$
\begin{equation}
\psi_{0,k}(x,y) := \sqrt{{2 \over \ell_{1}}}\,\cos(k\pi x/\ell_{1}) \, {\cosh(k\pi(\ell_{2}-y)/\ell_{1}) \over \cosh(k\pi\ell_{2}/\ell_{1})}\, , \qquad \mu_{0,k} := k\tanh(k\pi\ell_{2}/\ell_{1})\, ,
\end{equation}
and
\begin{equation} 
\psi_{1,k}(x,y) := \sqrt{{2 \over \ell_{1}}}\, \cos(k\pi x/\ell_{1}) \, {\sinh(k\pi y/\ell_{1}) \over \sinh(k\pi\ell_{2}/\ell_{1})}\,, \qquad \mu_{1,k} := k\cotanh(k\pi\ell_{2}/\ell_{1}) .
\end{equation} 
If $f_0 \in H^{1/2}(\Gamma_0)$ and $g_0 \in H^{-1/2}(\Gamma_0)$ are given, for $k \geq 0$ integer denote 
\begin{equation} \label{eq:Def-fg0k}
f_{0,k} := \int_0^{\ell_{1}} f_0(x)\psi_{0,k}(x,0)dx, \qquad
g_{0,k} := \int_0^{\ell_{1}} g_0(x)\psi_{0,k}(x,0)dx\, .
\end{equation}
Then the following equation 
\begin{equation}\label{eq:CL1-gene}
\begin{cases}
-\Delta u = 0 & \mbox{in }\, \Omega \cr
u = f_0 & \mbox{on }\, \Gamma_0 \cr
\displaystyle {\partial u \over \partial {\bf n} } = g_0 & \vphantom{{\partial u \over \partial {\bf n} } }\mbox{on }\, \Gamma_0\cr
\displaystyle {\partial u \over \partial {\bf n} } = 0 & \vphantom{{\partial u \over \partial {\bf n} } }\mbox{on }\, \Gamma.\cr
\end{cases}
\end{equation}
has a unique solution $u \in H^1(\Omega )$ if, and only if, the
Cauchy boundary data $f_0,g_0$ satisfy the following compatibility condition:
\begin{equation} \label{eq:CLCNS-gene}
\sum_{k=1}^{\infty }k \sinh^2(k\pi\ell_{2}/\ell_{1}) \left(f_{0,k} - {\ell_{1}\cotanh(k\pi\ell_{2}/\ell_{1})\over k \pi}\, g_{0,k} \right)^2 <
\infty .
\end{equation}
Moreover, when the above condition is satisfied we have
\begin{equation*}%\label{eq:Expression-u}
u = \sum_{k=0}^\infty f_{0,k}\psi_{0,k} - g_{0,0}\psi_{0,0} + 
\sum_{k=1}^\infty {\sinh^2(k\pi\ell_{2}/\ell_{1}) \over \cosh(k\pi\ell_{2}/\ell1)}\left(f_{0,k} - 
{\ell_{1}\cotanh(k\pi\ell_{2}/\ell_{1}) \over k\pi}g_{0,k}\right)\psi_{1,k}\, ,
\end{equation*}
and there exist two positive constants $c_1,c_2$ such that
if we denote by  $\|(f_{0},g_{0})\|_{*}^2$ the quantity
\begin{equation}\label{eq:Norm-Cauchy-data}
|f_{0,0}|^2 + |g_{0,0}|^2  + 
\sum_{k=1}^{\infty }k|f_{0,k}|^2 + 
\sum_{k=1}^\infty k\sinh^2(k\pi\ell_{2}/\ell_{1})
\left(f_{0,k} - {\ell_{1}\cotanh(k\pi\ell_{2}/\ell_{1})\over k\pi }\, g_{0,k} \right)^2 \nonumber 
\end{equation}
then we have 
\begin{equation}\label{eq:EstimCL-gene}
c_{1}\, \|(f_{0},g_{0})\|_{*}^2 
\leq  \|u\|_{H^1(\Omega)}^2 \leq 
c_{2}\, \|(f_{0},g_{0})\|_{*}^2 .
\end{equation}
\end{corollary}

\begin{remark}
It is important to point out a particular feature of the above compatibility condition \eqref{eq:CLCNS-gene} on the Cauchy boundary data $f_{0} \in H^{1/2}(\Gamma_{0})$ and $g_{0} \in H^{-1/2}(\Gamma_{0})$. 
Indeed, this condition contains in a hidden and subtle way the homogeneous Neumann boundary condition $\partial u/\partial{\bf n} = 0$ on $\Gamma$, through the Steklov eigenfunctions $\psi_{j,k}$, and therefore the Fourier coefficients $f_{0,k}$ and $g_{0,k}$.

To be more specific, for a real parameter $\alpha \in [0,1]$ let us consider the following equation with Cauchy boundary data on $\Gamma_{0}$: find $u\in H^1(\Omega)$ such that
\begin{equation}\label{eq:CL1-alpha}
\begin{cases}
-\Delta u = 0 & \mbox{in }\, \Omega \cr
u = f_0 & \mbox{on }\, \Gamma_0 \cr
\displaystyle {\partial u \over \partial {\bf n} } = g_0 & \vphantom{{\partial u \over \partial {\bf n} } }\mbox{on }\, \Gamma_0\cr
\displaystyle \alpha {\partial u \over \partial {\bf n} } + (1-\alpha) u = 0 & \vphantom{{\partial u \over \partial {\bf n} } }\mbox{on }\, \Gamma.\cr
\end{cases}
\end{equation}
For $\alpha = 0$ we have a homogeneous Dirichlet boundary condition on $\Gamma$, while for $\alpha = 1$ the boundary condition is of Neumann type. For $0 < \alpha < 1$, the boundary condition on $\Gamma$ is a Fourier boundary condition (sometimes named Robin boundary condition). 

This equation can be solved in a similar manner by considering the following Steklov eigenfunctions. For $j= 0$ or $j = 1$ consider the functions $\psi_{j,k}$ solutions to the Steklov eigenvalue problem
\begin{equation}\label{eq:Steklov-j-alpha}
\begin{cases}
-\Delta \psi_{j,k} = 0 & \mbox{in }\, \Omega \cr
\displaystyle {\partial \psi_{j,k} \over \partial {\bf n} } = \mu_{j,k}\psi_{j,k} & \vphantom{{\partial u \over \partial {\bf n} } }\mbox{on }\, \Gamma_j\cr
\psi_{j,k} = 0 & \mbox{on }\, \Gamma_{1-j} \cr
\displaystyle \alpha {\partial \psi_{j,k} \over \partial {\bf n} } + (1 - \alpha) \psi_{j,k} = 0 & \vphantom{{\partial u \over \partial {\bf n} } }\mbox{on }\, \Gamma \cr
\end{cases}
\end{equation}
with the usual normalization $\int_{\Gamma_{j}}\psi_{j,k}(\sigma)\psi_{j,k'}(\sigma)\,d\sigma = \delta_{k'k}$. 
It is not difficult to see that $\psi_{1,k}(x,y) = \psi_{0,k}(x,\ell_{2}-y)$ and that 
$$\psi_{0,k}(x,y) = \phi_{k}(x) {\sinh(\sqrt{\lambda_{k}}\,(\ell_{2}-y)) \over \sinh(k\ell_{2})}$$
where $\phi_{k}$ solves the Sturm--Liouville eigenvalue problem on $(0,\ell_{1})$:
\begin{equation} 
\begin{cases}
-\phi''_{k} = \lambda_{k}\phi_{k} & \mbox{in }\, (0,\ell_{1}) \cr
-\alpha\phi'_{k}(0) + (1-\alpha) \phi_{k}(0) = 0 & \cr
\alpha\phi'_{k}(\ell_{1}) + (1-\alpha)\phi_{k}(\ell_{1}) = 0. & \cr
\end{cases}
\end{equation}
Now it is clear that $\lambda_{k}$ depends on $\alpha$ and that $\lambda_{k} \sim c_{*}(\ell_{1}) k^2$ as $k\to +\infty$. One may compute also the Steklov eigenvalues $\mu_{j,k}$ which are
$$\mu_{0,k} = \mu_{1,k} := \mu_{k} := \sqrt{\lambda_{k}}\, \cotanh(\sqrt{\lambda_{k}}\,\ell_{2}).$$
We then define the coefficients $f_{0,k},g_{0,k}$ as in \eqref{eq:Def-fg0k}, with the new eigenfunctions $\psi_{0,k}$, and in a manner strictly identical to what we have seen above, one finds that a necessary and sufficient condition on $(f_{0},g_{0})$ is given by 
\begin{equation}
\label{eq:CLCNS-alpha}
\sum_{k=1}^{\infty }\sqrt{\lambda_{k}}\, \cosh^2(\sqrt{\lambda_{k}}\,\ell_{2}) \left(f_{0,k} - {\tanh(\sqrt{\lambda_{k}}\,\ell_{2})\over \sqrt{\lambda_{k}}}\, g_{0,k} \right)^2 <
\infty .
\end{equation}
Denote also by $\|(f_{0},g_{0})\|_{*,\alpha}$ the norm defined by the quantity 
$$\|(f_{0},g_{0})\|_{*,\alpha}^2 := \sum_{k=1}^{\infty }\sqrt{\lambda_{k}}\,|f_{0,k}|^2 + 
\sqrt{\lambda_{k}}\,\cosh^2(\sqrt{\lambda_{k}}\,\ell_{2})
\left(f_{0,k} - {\tanh(\sqrt{\lambda_{k}}\,\ell_{2})\over \sqrt{\lambda_{k}}}\, g_{0,k} \right)^2,  $$
which can be considered as a norm defined by an appropriate scalar product, at least when $0\leq \alpha < 1$ (for $\alpha = 1$ one has to add the constants $|f_{0,0}|^2 + |g_{0,0}|^2$ as in \eqref{eq:Norm-Cauchy-data}).
Let ${\Bbb H}_{\alpha}$ be the space
\begin{equation}
{\Bbb H}_{\alpha} := \left\{ 
(f_{0},g_{0}) \in H^{1/2}(\Gamma_{0}) \times H^{-1/2}(\Gamma_{0}) \; ; \;
\|(f_{0},g_{0})\|_{*,\alpha} < \infty \right\}.
\end{equation}
One may verify that for each $\alpha$ the space ${\Bbb H}_{\alpha}$ endowed with the norm 
$$(f_{0},g_{0}) \mapsto \left( \|f_{0}\|_{H^{1/2}(\Gamma_{0})}^2 + \|(f_{0},g_{0})\|_{*,\alpha}^2\right)^{1/2}$$
is a Hilbert space. 

Now we claim that if $0\leq \alpha_{1} < \alpha_{2} \leq 1$, we have ${\Bbb H}_{\alpha_{1}} \cap {\Bbb H}_{\alpha_{2}} = \{0\}$. Indeed, if $(f_{0},g_{0}) \in {\Bbb H}_{\alpha_{1}} \cap {\Bbb H}_{\alpha_{2}}$, and $u_{1}$ and $u_{2}$ solve \eqref{eq:CL1-alpha} for the values $\alpha_{1}$ and $\alpha_{2}$ respectively, then $v := u_{1} - u_{2}$ satisfies $\Delta v = 0$ in $\Omega$ and $v = \partial v/\partial{\bf n} = 0$ on $\Gamma_{0}$. The unique continuation principle implies that $v \equiv 0$ in $\Omega$, and in particular $\partial v/\partial{\bf n} = v = 0$ on $\Gamma$, that is we have
$$\alpha_{1}{\partial u_{1} \over \partial{\bf n}} + (1 - \alpha_{1}) u_{1} = 0, \qquad
\alpha_{2}{\partial u_{1} \over \partial{\bf n}} + (1 - \alpha_{2}) u_{1} = 0 \qquad \mbox{on }\, \Gamma.
$$
Since $\alpha_{1} \neq \alpha_{2}$, we infer that $\partial u_{1}/\partial{\bf n} = u_{1} = 0$ on $\Gamma$, and again by the unique continuation principle we have $u_{1} \equiv 0$ on $\Omega$, which yields $f_{0} = g_{0} = 0$.

One sees that the space of Cauchy datas $(f_{0},g_{0})$ on $\Gamma_{0}$ for which one may solve uniquely, and in a well-posed manner the Cauchy problem
$$-\Delta u = 0 \quad\mbox{in }\, \Omega, \qquad
u = f_{0}\quad \mbox{and}\quad {\partial u \over \partial{\bf n}} = g_{0}\quad \mbox{on }\,\Gamma_{0}$$
contains ${\Bbb H}_{\alpha}$ for all $0 \leq \alpha \leq 1$. However it is also clear that one can exhibit many other subspaces of compatible Cauchy datas in order to solve the above equation. \qed
\end{remark}

%%%%%%%%%%%%%%%%%%%%%%%

%%%%%%%%%%%%%%%%%%%%%%%
\section{Remarks on numerical simulations}\label{sec:Rem-numerical}
We should mention that, according to the construction given in this paper, numerical simulations have been performed and other numerical experiments are under way.

The construction here as well as in \cite{GlHo08} or \cite{GlHo10} is of an open-loop type. Indeed, the construction of the vector field $X$ is given a priori, and is the basis of all the subsequent analysis.

For the time being, numerical simulations based on this open-loop control fail to be satisfactory. To compensate the open-loop strategy, G. Legendre and F-X. Vialard (both from Universit\'e Paris-Dauphine, France) have suggested to compute a new vector field $X$ at each time step. However, despite the fact that significant improvements were made, the instability pointed out in the previous sections seems to play a crucial role in the numerical difficulties in the tracking of the motion of the curve $\gamma_0$.

Other options in order to compensate these behaviours are currently being experimented, and their description and mathematical treatment are postponed to future reports and papers.

\medskip

\bibliographystyle{plain}
\bibliography{Biblio-horsin}

\end{document}